%% file: Paper.tex
\documentclass[final, 12pt]{article}

\usepackage{amsmath}
\usepackage{amsfonts}
\usepackage{makeidx}
\usepackage{amsthm}
\usepackage{mathtools}
\usepackage{amssymb}
\usepackage{bm}
\usepackage{cite}
\usepackage{showkeys}
\usepackage[english]{babel}
\usepackage{dblaccnt}
\usepackage{accents}
\usepackage{graphicx}
\usepackage{subfig}
\usepackage{psfrag}\usepackage{color}
\usepackage{color}
\usepackage{amscd}
\usepackage[mathscr]{euscript}
\usepackage{lipsum}
\usepackage{epstopdf}
\usepackage{algorithm}
\usepackage{algpseudocode}

\input{mathmacros.tex}

\newtheorem{defi}{Definition}
\newtheorem{lemma}[defi]{Lemma}
\newtheorem{theorem}[defi]{Theorem}

\newtheorem{assumption}[defi]{Assumption}

\usepackage{fancyhdr}
\usepackage{lipsum}

\begin{document}

\sloppy

\title{Imaging of bi-anisotropic periodic structures from electromagnetic near field data}
\author{Dinh-Liem Nguyen\thanks{Department of Mathematics, Kansas State University, Manhattan, KS 66506; (\texttt{dlnguyen@ksu.edu})} \and Trung Truong\thanks{Department of Mathematics, Kansas State University, Manhattan, KS 66506; (\texttt{trungt@ksu.edu})}
}
\date{}
\maketitle

\begin{center}
\vspace{-0.6cm}
\textit{Dedicated to Professor  Michael Klibanov on the occasion of his 70th birthday}
\end{center}

\begin{abstract}
This paper is concerned with the inverse  scattering problem for the three-dimensional Maxwell's equations in  bi-anisotropic periodic structures. The inverse scattering problem aims to determine the shape of bi-anisotropic 
periodic scatterers from electromagnetic near field data at a fixed frequency. The Factorization method is studied as an analytical and numerical tool for solving the inverse problem. We provide a rigorous justification of the Factorization method which results in the unique determination  and a fast imaging algorithm for  the periodic scatterer. Numerical examples for imaging three-dimensional periodic structures  are presented to examine the efficiency of the method.

\end{abstract}


\bigskip


\section{Introduction}
We consider an inverse  scattering problem which aims to determine the shape of 
 bi-anisotropic periodic structures from electromagnetic near field data.  The Factorization  method  is studied to solve
 the inverse scattering problem. This study is part   of the research on 
 inverse scattering from periodic structures. The periodic structures of interests here  are mainly motivated by the one-dimensional and two-dimensional photonic crystals~\cite{Dorfl2012}.  This research topic  has  received a great attention  during the past years thanks to its potential applications in nondestructive testing and optimal design in optics~\cite{Bao2001}.
 
  There has been a large body of literature covering theoretical studies on uniqueness and stability along with numerical reconstruction methods. Since we are  interested in  numerical reconstructions, we will mainly discuss  related results in this direction. 
 We first refer to~\cite{Arens2003a, Arens2005, Lechl2010, Elsch2012,Yang2012, Nguye2014, Zhang2014, Hadda2017, Lechl2018, Nguye2020, Harri2020} and references therein for an unexhausted list of results in the case of  scalar equations of Helmholtz type.  However, to our knowledge, there has been only a limited number of studies on numerical reconstructions for  the case of full Maxwell's equations, see~\cite{Sandf2010, Lechl2013b, Bao2014, Nguye2016, Jiang2017}. The numerical methods studied in these papers are the Factorization method~\cite{Sandf2010, Lechl2013b, Nguye2016} and the near field imaging method  that relies  on a transformed field expansion~\cite{Bao2014, Jiang2017}. The main advantages of the near field imaging method  are that it requires scattering data associated with only one incident plane wave and can provide  super-resolved resolution. However, the analysis of this method assumes that the periodic scattering  layer is described by  a smooth periodic function multiplied by a small surface deformation parameter.  Although the Factorization method requires scattering data generated by multiple incident plane waves, its analysis can allow us to recover periodic scattering structures of arbitrary shape. This method belongs to the class of sampling or qualitative methods that were introduced by D. Colton and A. Kirsch~\cite{Colto1996, Kirsc1998}. The Factorization method aims to construct a necessary and sufficient characterization of the unknown scatterer from multi-static data. This characterization can serve as a fast and simple imaging algorithm. We refer to~\cite{Kirsc2008} for more details about the Factorization method.  
 
 The Factorization method is studied in this paper as  an analytical and numerical tool for solving the inverse  scattering problem
 for the full Maxwell's equations in bi-anisotropic periodic structures. This study is most related to~\cite{Nguye2016} which considers the Maxwell's equations in chiral periodic structures. The chiral and bi-anisotropic media belong to the class of complex electromagnetic media which has recently received a considerable attention thanks to their applications in photonics and nano-optics, see~\cite{Roach2012, Macka2010} and references therein.
 The Maxwell's equations describing the propagation of  electromagnetic waves  through chiral or bi-anisotropic media are coupled with the constitutive relations (see~\eqref{cr1}) that are more complicated than those of standard media (e.g. non-magnetic media). Note that except~\cite{Nguye2016}  all the previously cited works for the Maxwell's equations  consider  non-magnetic media. Unlike the scalar coefficients in the chiral media case~\cite{Nguye2016}, the  coefficients in the bi-anisotropic case 
 considered in the present work are all matrix-valued functions. Therefore, the analysis of the Factorization method for the corresponding inverse problem of interest is technically more complicated. Under some assumption (Assumption~\ref{assum}) corresponding to the absorbing material case and the smallness of one of the coefficients, we can prove the coercivity (Lemma~\ref{coercive}) of the middle operator in the Factorization method. This is also the key   ingredient in the justification of the Factorization method.  
 The justification  remains open if the Assumption~\ref{assum} does not hold true (e.g. the coefficients are all real-valued). 
 
 The paper is organized as follows. In Section 2 we formulate the direct scattering problem and its equivalent integro-differential equation. The inverse problem of interest is formulated in Section 3. Section 4 provides a characterization of the scattering domain via the range of some operator.  Section 5 is dedicated to the analysis of the Factorization method in which we prove the main theorems of the paper. Finally we present in Section 6 some numerical examples for imaging three-dimensional periodic structures using the Factorization method.

\section{The direct problem}
We consider a three-dimensional periodic structure which is infinitely $2\pi$-periodic in $x_1$, $x_2$ and bounded in $x_3$, for example see Figure~\ref{fig1}(c) (here $x_1$, $x_2$, $x_3$ are the components of a vector $\mathbf{x}=(x_1,x_2,x_3)^\top$ in $\mathbb{R}^3$).  Assume that the medium inside the  periodic structure is inhomogeneous and bi-anisotropic  and the outside medium is homogeneous. 
We denote that the electric field $\mathbf{E}$, the magnetic field $\mathbf{H}$, the electric flux density $\mathbf{D}$ and the magnetic flux density $\mathbf{B}$ are three-dimensional vector-valued functions. The scattering of time-harmonic electromagnetic waves (with positive frequency $\omega$) from the bi-anisotropic periodic structure 
is described by the Maxwell's equations 
\begin{equation}\label{me1}
\curl \mathbf{E} +i\omega \mathbf{B} = 0,\quad \curl \mathbf{H} - i\omega \mathbf{D} = 0, \quad \text{in } \R^3,
\end{equation}
 along with the constitutive relations 
 \begin{equation}\label{cr1}
\mathbf{B} = \mu \mathbf{H}+ \xi\sqrt{\varepsilon_0\mu_0}\ \mathbf{E}, \quad \mathbf{D} = \varepsilon \mathbf{E} + \overline{\xi}\sqrt{\varepsilon_0\mu_0}\ \mathbf{H}.
\end{equation}
Here  $\varepsilon$ and $\mu$ are respectively the  permittivity and permeability of the scattering medium, and
 the parameter $\xi$ is typically described as $\xi =  \chi + i\kappa$ 
where $\chi$ is the chirality parameter and $\kappa$ is the non-reciprocity 
parameter of the medium (see \cite{Macka2010}).  These are  $3\times 3$ matrix-valued bounded functions satisfying  $\varepsilon = \varepsilon_0I_3, \mu = \mu_0I_3, \xi = 0I_3$ in the  outside  medium for some positive constants $\varepsilon_0$ and $\mu_0$ ($I_3$ is the $3\times 3$ identity matrix). 
We introduce the relative quantities
$$
\epsr = \frac{\varepsilon}{\varepsilon_0},\quad \mur = \frac{\mu}{\mu_0}
$$
and the scaled quantities (with the same notations as the original ones)
$$
\mathbf{E} = \sqrt{\varepsilon_0}\ \mathbf{E}, \quad \mathbf{H} = \sqrt{\mu_0}\ \mathbf{H}.
$$
Using these new quantities and plugging (\ref{cr1}) into (\ref{me1}) we obtain
\begin{equation}\label{me2}
\curl \mathbf{E} - ik(\mur\mathbf{H} + \xi\mathbf{E}) = 0,\quad 
\curl \mathbf{H} + ik(\epsr\mathbf{E} + \overline{\xi}\mathbf{H}) = 0
\end{equation}
where $k=\omega\sqrt{\varepsilon_0\mu_0}$ is the wave number. 
Assuming that $\mur$ is invertible almost everywhere in $\mathbb{R}^3$, we can write  the first equation in (\ref{me2}) as 
\begin{equation}\label{EH}
\mathbf{H} = -\frac{i}{k}\mur^{-1}\curl \mathbf{E} - \mur^{-1}\xi\mathbf{E}.
\end{equation}
Plugging this into the second equation of~\eqref{me2}  and rearranging the resulting equation we obtain
\begin{equation}\label{me3}
\curl (\mur^{-1}\curl \mathbf{E}) + ik\left[ \overline{\xi}\mur^{-1}\curl \mathbf{E} - \curl (\mur^{-1}\xi\mathbf{E}) \right] -k^2(\epsr - \overline{\xi}\mur^{-1}\xi)\mathbf{E} = 0.
\end{equation}
Now assume that the periodic structure is illuminated by some incident field $\mathbf{E}^{in}$ satisfying 
$$
\curl \curl \mathbf{E}^{in}-k^2 \mathbf{E}^{in} = 0,
$$
%
then there arises the scattered field $\u$ defined by $\u = \E - \E^{in}$. We can
thus rewrite (\ref{me3}) for the scattered field $\u$ as
\begin{multline}
\label{u_sc}
\curl^2 \mathbf{u}-k^2\mathbf{u} = k^2\left[ (P - \overline{\xi}\mur^{-1}\xi)(\mathbf{E}^{in}+\mathbf{u}) + \frac{i}{k}\overline{\xi}\mur^{-1}(\curl \mathbf{E}^{in}+\curl \mathbf{u})\right]\\
+\curl\left[ Q(\curl \mathbf{E}^{in} + \curl \mathbf{u}) -ik\mur^{-1}\xi(\mathbf{E}^{in}+\mathbf{u}) \right]
\end{multline}
where $P$ and $Q$ are the contrasts defined by 
$$P=\epsr - I_3,\quad Q= I_3 - \mur^{-1}.$$ 
Note that $P$ and $Q$ are supported inside the periodic structure. 
We now define that for $\alpha = (\alpha_1,\alpha_2,0)^\top \in \mathbb{R}^3$, a function $\mathbf{v}: \mathbb{R}^3 \to \mathbb{R}^3$ is called \emph{$\alpha$-quasiperiodic} if for any $n = (n_1,n_2,0)^\top\in \mathbb{Z}^3$
$$
\mathbf{v}(x_1+n_12\pi,x_2+n_22\pi,x_3) = e^{2\pi i\alpha\cdot n}\mathbf{v}(x_1,x_2,x_3),\quad \mathbf{x}\in\mathbb{R}^3.
$$
Following the typical approach for periodic scattering problems we consider incident fields which are  $\alpha$-quasiperiodic plane waves (see~\eqref{u_in} for the incident plane waves used to generate the data for the inverse problem.)

For  $\alpha =(\alpha_1,\alpha_2,0)^\top$ we denote
\begin{equation}
\label{alpha}
\alpha_m = (\alpha_1+m_1,\alpha_2+m_2,0)^\top,\quad m = (m_1,m_2)^\top\in\mathbb{Z}^2. 
\end{equation}
Then it is well known that the scattered field $\u$ is also $\alpha$-quasiperiodic, see for instance~\cite{Schmi2003}. Thus the scattering problem can be reduced to one period $\Omega$ of the periodic structure, which is defined by
$$ \Omega = (-\pi,\pi)^2 \times \mathbb{R}.$$
Let $D \subset \Omega$ be an open  set defined by 
$$
\ol{D} = [\mathrm{supp}(Q) \cup \mathrm{supp}(P)] \cap \Omega.
$$  
%
We now complete the direct scattering problem by the well-known Rayleigh expansion radiation condition for the scattered field. We first need some notations. 
Let $h$ be a positive constant such that 
\begin{equation}\label{h}
h>\sup\{|x_3|: \mathbf{x} \in \ol{D}\},
\end{equation}
and for $m\in\mathbb{Z}^2$ we define
\begin{equation}\label{beta}
\beta_m = \left\{ \begin{array}{rl} \sqrt{k^2-|\alpha_m|^2},& |\alpha_m| \leq k\\
i\sqrt{|\alpha_m|^2-k^2},&   |\alpha_m| > k \end{array}\right..
\end{equation}
Then $\mathbf{u}$ is called radiating if it can be expressed as the following Rayleigh expansion
\begin{equation}
\label{rc}
\mathbf{u}(\mathbf{x}) = \sum_{m\in\mathbb{Z}^2}\widehat{\mathbf{u}}^{\pm}_me^{i(\alpha_m\cdot \mathbf{x} \pm \beta_m(x_3\mp h))}\quad \text{for}\ x_3 \gtrless \pm h,
\end{equation}
where $\widehat{\mathbf{u}}^{\pm}_m$ are the Rayleigh coefficients of $\u$ defined by
$$
\widehat{\mathbf{u}}^{\pm}_m = \frac{1}{4\pi^2}\int_{-\pi}^{\pi}\int_{-\pi}^{\pi}e^{-i\alpha_m\cdot \mathbf{x}}\, \mathbf{u}(x_1,x_2,\pm h)\d x_1\d x_2,\quad m\in\mathbb{Z}^2.
$$
Note that all but finitely many terms in (\ref{rc}) are exponentially decaying, which helps us easily deduce pointwise absolute convergence of the series. Moreover, we need $\beta_m$ to be nonzero for all $m\in\mathbb{Z}^2$ or $k$ is not a Wood's anomaly.  The technical reason behind this assumption is that the representation of the $\alpha$-quasiperiodic  Green's function we use in~\eqref{green}  is not well-defined at a Wood's anomaly.

The direct problem inlcuding the equation~\eqref{u_sc} and the radiation condition \eqref{rc} can be reformulated as an integro-differential equation. This formulation  will be useful for the analysis of the inverse problem. First we define 
$$
H_{\alpha,\loc}(\curll,\Omega) = \{\mathbf{u}\in H_{\loc}(\curll,\Omega): \mathbf{u} = \widetilde{\mathbf{u}}|_{\Omega}\ \text{for some}\ \alpha \text{-quasiperiodic}\ \widetilde{\mathbf{u}}\in H_{\loc}(\curll,\mathbb{R}^3)	\}.
$$
We consider the equation (\ref{u_sc}) in its more general form
\begin{multline}\label{E1g}
\curl^2 \mathbf{u}-k^2\mathbf{u} = k^2\left[ (P - \xi\mur^{-1}\xi)(\mathbf{g}+\mathbf{u}) + \frac{i}{k}\xi\mur^{-1}(\mathbf{f}+\curl \mathbf{u})\right]\\
+\curl\left[ Q(\mathbf{f} + \curl \mathbf{u}) -ik\mur^{-1}\xi(\mathbf{g}+\mathbf{u}) \right]
\end{multline}
where $\mathbf{f}$ and $\mathbf{g}$ are some generic functions in $L^2(D,\mathbb{C}^3)$. It is clear that (\ref{u_sc}) is a particular case of (\ref{E1g}) when $\mathbf{f} = \curl\mathbf{E}^{in}$ and $\mathbf{g} = \mathbf{E}^{in}$.
Let
\begin{align*}
\mathcal{S}(\mathbf{u},\mathbf{f},\mathbf{g}) &= (P - \xi\mur^{-1}\xi)(\mathbf{g}+\mathbf{u}) + \frac{i}{k}\xi\mur^{-1}(\mathbf{f}+\curl \mathbf{u}), \\
 \mathcal{T}(\mathbf{u},\mathbf{f},\mathbf{g}) &= Q(\mathbf{f} + \curl \mathbf{u}) -ik\mur^{-1}\xi(\mathbf{g}+\mathbf{u}).
\end{align*}
It is clear that $\mathcal{S}(\mathbf{u},\mathbf{f},\mathbf{g})$ and $\mathcal{T}(\mathbf{u},\mathbf{f},\mathbf{g})$ are compactly supported in $D$ and (\ref{E1g}) can be written as
\begin{equation}\label{equ}
\curl^2 \mathbf{u}-k^2\mathbf{u} = k^2\mathcal{S}(\mathbf{u},\mathbf{f},\mathbf{g}) + \curl \mathcal{T}(\mathbf{u},\mathbf{f},\mathbf{g}).
\end{equation}
Now let us consider the variational problem of finding $\mathbf{u}\in H_{\alpha,\loc}(\curll,\Omega)$ such that 
\begin{equation}\label{var}
\int_{\Omega}(\curl \mathbf{u}\cdot\curl\overline{\mathbf{v}} - k^2\mathbf{u}\cdot\overline{\mathbf{v}})\ \d\mathbf{x} = k^2\int_{D}\mathcal{S}(\mathbf{u},\mathbf{f},\mathbf{g})\cdot\overline{\mathbf{v}}\ \d\mathbf{x} + \int_{D}\mathcal{T}(\mathbf{u},\mathbf{f},\mathbf{g})\cdot\curl \overline{\mathbf{v}}\ \d\mathbf{x}
\end{equation}
for all $\mathbf{v}\in H_{\alpha}(\curll,\Omega)$ with compact support. 

It is similar to~\cite{Nguye2015} that the
variational problem~\eqref{var} can be equivalently reformulated as
\begin{equation}
\label{indif}
\mathbf{u} = \mathcal{A}\mathcal{S}(\mathbf{u},\mathbf{f},\mathbf{g}) + \mathcal{B}\mathcal{T}(\mathbf{u},\mathbf{f},\mathbf{g})\quad \text{in}\ \Omega,
\end{equation}
where $\mathcal{A}$ and  $\mathcal{B}$ are bounded linear operators from $L^2(D,\mathbb{C}^3)$ to $H_{\alpha,\loc}(\curll,\Omega)$
defined by 
$$\mathcal{A}\mathbf{h}(\x) =(k^2+\nabla \div ) \int_{D}G_k(\mathbf{x}-\mathbf{y})\mathbf{h}(\mathbf{y})\d\mathbf{y} \quad \text{and}\quad \mathcal{B}\mathbf{h}(\x) = \curl \ \int_{D}G_k(\mathbf{x}-\mathbf{y})\mathbf{h}(\mathbf{y})\d\mathbf{y},
$$ 
where  $G_k$ is the $\alpha$-quasiperiodic Green's function of the three-dimensional Helmholtz equation
\begin{equation}\label{green}
G_k(\mathbf{x}) = \frac{i}{8\pi^2}\sum_{m\in\mathbb{Z}^2}\frac{1}{\beta_m}e^{i(\alpha_m\cdot \mathbf{x} + \beta_m|x_3|)}, \quad x \in \Omega, x_3 \neq 0.
\end{equation}
The proof of the Fredholm property of the integro-differential equation~\eqref{indif} can be done similarly as in~\cite[Theorem 4]{Nguye2019} under the following  assumption.
 \begin{assumption}
 \label{th:positive}
  Assume that $D$ is a Lipschitz domain and that $\epsr, \mur, \mur^{-1}, \xi \in L^\infty(\Omega, \C^{3\times 3})$ are symmetric almost everywhere in $\R^3$, $\xi$ is real-valued. Furthermore, assume that
 there exist positive constants $\gamma_1, \gamma_2$ such that 
for any $\mathbf{a} \in \mathbb{C}^3$
\begin{align*}
\Re(\mur^{-1}\mathbf a \cdot \overline{\mathbf a}) \geq \gamma_1|\mathbf a|^2, \quad \Re((\epsr - \xi\mur^{-1}\xi)\mathbf a\cdot \overline{\mathbf{a}}) \geq \gamma_2|\mathbf a|^2, \quad \| |\mur^{-1}\xi|_F\|_{L^\infty}< \gamma_1\gamma_2,
\end{align*}
 almost everywhere in $\R^3$.
  \end{assumption}
 Here  $|\cdot|_F$ is the Frobenius matrix norm. The Fredholm property is actually valid for  any matrix norm in the last constraint in this assumption. However,  the Frobenius norm is used here for the convenience of the proof of Lemma~\ref{coercive} for the Factorization method analysis.   The uniqueness of solution is out of the scope of this paper since  the paper aims to solve the corresponding inverse scattering problem. Typically the well-posedness of periodic scattering problems holds for all but a discrete set of wave numbers $k$, see for instance~\cite{Schmi2003}.   Therefore, for the remaining part of the paper we  assume we work with the wave number $k$  such that the direct problem or the equivalent integro-differential equation~\eqref{indif} is well-posed.

\section{The inverse problem}

In this section we formulate the inverse  problem we want to solve. In addition to Assumption 1, the following assumption is important to the study of the inverse problem.
\begin{assumption}\label{assum}
We assume  that $\Omega \setminus \ol{D}$ has at most two connected components and that each connected component of $\Omega \setminus \ol{D}$ is unbounded. Moreover, there exist $C_1, C_2>0$ such that for all $\mathbf{z}\in\mathbb{C}^3$
$$
-\Im(\mur^{-1})\mathbf{z}\cdot \bar{\mathbf{z}} \geq C_1|\mathbf{z}|^2, \qquad \Im(\epsr - \xi\mur^{-1}\xi)\mathbf{z}\cdot \bar{\mathbf{z}} \geq C_2|\mathbf{z}|^2
$$
almost everywhere in $D$ and that
$$
\frac{1}{2}\left(\lVert |\mur^{-1}\xi|_F \rVert_{L^\infty}^2+1\right) \leq \min\left\{ C_1,C_2 \right\}.
$$
\end{assumption}
Here is an example of the parameters satisfying both Assumptions 1 and 2
\begin{align}
\label{coeff1}
\epsr &= \begin{bmatrix} 1+0.75i &0&0 \\ 0&1+0.9i&0\\ 0&0&1+0.8i \end{bmatrix},\quad \mur^{-1} = \begin{bmatrix} 1-0.7i &0&0 \\ 0&1-i&0\\ 0&0&1-0.9i \end{bmatrix}, \\
\xi &= \begin{bmatrix} 0.01 &0&0 \\ 0&0.02&0\\ 0&0&0.05 \end{bmatrix} \quad \text{in } D, \quad \text{and }
\epsr = \mur^{-1} =I_3,\quad \xi = 0 \quad \text{in } \Omega \setminus \ol{D}.
\label{coeff2}
\end{align}
Since the direct problem is well-posed  we can define the solution operator $G:L^2(D,\mathbb{C}^3) \times L^2(D,\mathbb{C}^3) \to \ell^2(\mathbb{Z}^2,\mathbb{C}^4)$
$$
G(\mathbf{f},\mathbf{g}) = (\widehat{u}_{m,1}^+,\widehat{u}_{m,1}^-,\widehat{u}_{m,2}^+,\widehat{u}_{m,2}^-)_{m\in\mathbb{Z}^2}
$$
where $(\widehat{u}_{m,1}^+,\widehat{u}_{m,1}^-,\widehat{u}_{m,2}^+,\widehat{u}_{m,2}^-)_{m\in\mathbb{Z}^2}$ are the Rayleigh sequences of the first two components $(\widehat{\mathbf{u}}_m^{\pm})_{m\in\mathbb{Z}^2}$ of the radiating variational solution $\mathbf{u}$ of (\ref{E1g}).

For  $\x = (x_1,x_2,x_3)^\top$, we denote  $\tilde{\x} = (x_1,x_2,-x_3)^\top$.
For the inverse problem, we consider many incident plane waves as follows
\begin{equation}
\label{u_in}
\varphi_{m}^{(l)\pm} = p_m^{(l)}e^{i(\alpha_m\cdot\mathbf{x} + \beta_mx_3)} \pm  \tilde{p}_m^{(l)}e^{i(\alpha_m\cdot\mathbf{x} - \beta_mx_3)}, \quad l = 1,2,\ m\in\mathbb{Z}^2,
\end{equation}
where the polarizations $p_m^{(1)}$, $p_m^{(2)}$ are linearly independent vectors such that $ |p_m^{(1)}| = |p_m^{(2)}| = 1$ and $\varphi_{m}^{(l)\pm}$ are divergence-free. One possible choice could be
\begin{equation}\label{por}
p_m^{(1)} = \frac{(0,\beta_m,-\alpha_{m,2})^\top}{\sqrt{|\alpha_{m,2}|^2+|\beta_m|^2}},\quad p_m^{(2)} = \frac{(-\beta_m,0,\alpha_{m,1})^\top}{\sqrt{|\alpha_{m,1}|^2+|\beta_m|^2}}.
\end{equation}
Note that these incident plane waves were proposed in~\cite{Lechl2013b} for the analysis of the Factorization method for 
the Maxwell's equations in non-magnetic periodic structures. 
We now define the Herglotz operator  $H: \ell^2(\mathbb{Z}^2,\mathbb{C}^4) \to L^2(D,\mathbb{C}^3)\times  L^2(D,\mathbb{C}^3)$ that maps a sequence of coefficients $(a_m) = (a_{m,1}^+,a_{m,1}^-,a_{m,2}^+,a_{m,2}^-)_{m\in\mathbb{Z}^2}$ into
\begin{align}
\label{H}
H(a_m) = 
\begin{bmatrix}
\sum\limits_{m\in\mathbb{Z}^2}  \left(\frac{a_{m,1}^+}{\beta_mw_m^+}\curl\varphi_{m}^{(1)+} + \frac{a_{m,2}^+}{\beta_mw_m^+}\curl\varphi_{m}^{(2)+} + \frac{a_{m,1}^-}{\beta_mw_m^-}\curl\varphi_{m}^{(1)-} + \frac{a_{m,2}^-}{\beta_mw_m^-}\curl\varphi_{m}^{(2)-} \right) \\
\sum\limits_{m\in\mathbb{Z}^2} \left( \frac{a_{m,1}^+}{\beta_mw_m^+}\varphi_{m}^{(1)+} + \frac{a_{m,2}^+}{\beta_mw_m^+}\varphi_{m}^{(2)+} + \frac{a_{m,1}^-}{\beta_mw_m^-}\varphi_{m}^{(1)-} + \frac{a_{m,2}^-}{\beta_mw_m^-}\varphi_{m}^{(2)-}
\right)
\end{bmatrix}
\end{align}
where the weights
$$
w_m^+ = \left\{\begin{array}{ll} i,& |\alpha_m| \leq k \\ e^{-i\beta_mh}, & |\alpha_m|>k \end{array}\right., \quad w_m^- = \left\{\begin{array}{ll} 1,& |\alpha_m| \leq k \\ e^{-i\beta_mh}, & |\alpha_m|>k \end{array}\right.,\qquad m\in\mathbb{Z}^2
$$
are to help simplify the  calculations in the proofs of some analytical properties of $H$, see~\cite{Nguye2016}.  Here these weights are kept 
in the definition of $H$ for the convenience of the presentation. From~\cite{Nguye2016} we know that $H$ is a compact and injective operator.

We consider a near field measurement which is  motivated by  uniqueness results in~\cite{Kirsc1994} and applications in near field optics. Let $N:\ell^2(\mathbb{Z}^2,\mathbb{C}^4) \to \ell^2(\mathbb{Z}^2,\mathbb{C}^4)$ be the near field operator which maps a sequence $(a_m) \in\ell^2(\mathbb{Z}^2,\mathbb{C}^4)$ to the Rayleigh sequences of the first two components of the radiating variational solution $\u$ of (\ref{equ}) with $(\mathbf{f},\mathbf{g}) = H(a_m)$, that means
$$
N(a_m) = (\widehat{u}_{m,1}^+,\widehat{u}_{m,1}^-,\widehat{u}_{m,2}^+,\widehat{u}_{m,2}^-)_{m\in\mathbb{Z}^2}.
$$
 Now we are ready to state the inverse problem of interest.

\vspace{3mm}
\textbf{Inverse problem.} Given the near-field operator $N$, find the shape $D$ of the scatterer in  $\Omega$.

\section{A characterization of $D$}
In this section we will show that $D$ can be characterized by the range of the adjoint operator $H^*$ of the Herglotz operator $H$ defined in~\eqref{H}. Let us introduce two auxiliary operators which are $E:L^2(D,\mathbb{C}^3)\times L^2(D,\mathbb{C}^3) \to \ell^2(\mathbb{Z}^2,\mathbb{C}^4)$ and $W:\ell^2(\mathbb{Z}^2,\mathbb{C}^4)\to \ell^2(\mathbb{Z}^2,\mathbb{C}^4)$ defined by
$$
E(\mathbf{f},\mathbf{g}) = (\widehat{u}_{m,1}^+,\widehat{u}_{m,1}^-,\widehat{u}_{m,2}^+,\widehat{u}_{m,2}^-)_{m\in\mathbb{Z}^2}
$$
where $\mathbf{u}$ is the radiating variational solution to
\begin{equation}\label{E}
\curl^2\mathbf{u} - k^2\mathbf{u} = \curl\mathbf{f}+\mathbf{g}
\end{equation}
and 
$$
W(a_m) = 8\pi^2
\left[ \begin{array}{rrrr}
w_m^{*+}\overline{p_{m,1}^{(1)}} & w_m^{*+}\overline{p_{m,2}^{(1)}} & w_m^{*+}\overline{p_{m,1}^{(1)}} & w_m^{*+}\overline{p_{m,2}^{(1)}} \\
w_m^{*+}\overline{p_{m,1}^{(2)}} & w_m^{*+}\overline{p_{m,2}^{(2)}} & w_m^{*+}\overline{p_{m,1}^{(2)}} & w_m^{*+}\overline{p_{m,2}^{(2)}} \\
w_m^{*-}\overline{p_{m,1}^{(1)}} & w_m^{*-}\overline{p_{m,2}^{(1)}} & -w_m^{*-}\overline{p_{m,1}^{(1)}} & -w_m^{*-}\overline{p_{m,2}^{(1)}} \\
w_m^{*-}\overline{p_{m,1}^{(2)}} & w_m^{*-}\overline{p_{m,2}^{(2)}} & -w_m^{*-}\overline{p_{m,1}^{(2)}} & -w_m^{*-}\overline{p_{m,2}^{(2)}}
\end{array} \right]
\left[\begin{array}{cccc}
a_{m,1}^+\\ a_{m,1}^- \\ a_{m,2}^+ \\ a_{m,2}^-
\end{array}\right]
$$
where 
$$
w_m^{*+} = \left\{\begin{array}{ll} e^{-i\beta_mh},& |\alpha_m| \leq k \\ i, & |\alpha_m|>k \end{array}\right., \quad w_m^{*-} = \left\{\begin{array}{ll}  ie^{-i\beta_mh},& |\alpha_m| \leq k \\ i, & |\alpha_m|>k \end{array}\right.,\qquad m\in\mathbb{Z}^2
$$
and $p_m^{(1)}$, $p_m^{(2)}$ are given by (\ref{por}). Note that $E$ is well-defined thanks to well-posedness of (\ref{E}) for all frequencies $k>0$ (see \cite{Schmi2003}). In addition, we will extend $\mathbf{f}$ and $\mathbf{g}$ in (\ref{E}) by zero outside of $D$ if needed.

The proof of the following lemma can be done as in~\cite{Nguye2016}. 
\begin{lemma}
\label{H*}
The adjoint operator $H^*: L^2(D,\mathbb{C}^3)\times  L^2(D,\mathbb{C}^3) \to \ell^2(\mathbb{Z}^2,\mathbb{C}^4)$ satisfies
$$
H^* = WE.
$$
\end{lemma}
Next we introduce the $\alpha$-quasiperiodic Green's tensor 
$$
\mathbb{G}_k(\mathbf{x},\mathbf{z}) = G_k(\mathbf{x}-\mathbf{z})I_3 + \frac{1}{k^2}\nabla_{\mathbf{x}}\div_{\mathbf{x}}(G_k(\mathbf{x}-\mathbf{z})I_3),\quad \mathbf{x},\mathbf{z}\in \Omega,\ x_3 \neq z_3
$$
where $G_k$ is given by (\ref{green}) and $\nabla_{\mathbf{x}}$, $\div_{\mathbf{x}}$ are  taken componentwise and columnwise respectively. Note that for a fixed $\mathbf{z} \in \Omega$, $\mathbb{G}_k(\mathbf{x},\mathbf{z})$ solves 
\begin{equation}\label{fund}
\curl^2\mathbb{G}_k(\mathbf{x},\mathbf{z}) - k^2\mathbb{G}_k(\mathbf{x},\mathbf{z}) = \delta_{\mathbf{z}}(\mathbf{x})I_3
\end{equation}
in the sense of distribution and it satisfies the Rayleigh expansion radiation condition (the $\curl$ is taken columnwise).

Let $\mathbf{p} = (p_1,p_2,p_3)^\top \in \R^3$ and $\Psi_{\mathbf{z}}(\mathbf{x}) = k^2\mathbb{G}_k(\mathbf{x},\mathbf{z})\mathbf{p}$. Denote by 
$$
(\widehat{\Psi}_{\mathbf{z},m}) = (\widehat{\Psi}_{m,1}^+(\mathbf{z}),\widehat{\Psi}_{m,1}^-(\mathbf{z}),\widehat{\Psi}_{m,2}^+(\mathbf{z}),\widehat{\Psi}_{m,2}^-(\mathbf{z}))_{m\in\mathbb{Z}^2}
$$
the Rayleigh sequences of the first two components of $\Psi_{\mathbf{z}}$. Then $(\widehat{\Psi}_{\mathbf{z},m})$ can be explicitly given by
$$
(\widehat{\Psi}_{\mathbf{z},m}) = 
\begin{bmatrix}
(k^2 - \alpha_{m,1}^2)\widehat{G}^+_{k,m}(\mathbf{z})p_1 - \alpha_{m,1}\alpha_{m,2}\widehat{G}^+_{k,m}(\mathbf{z})p_2 - \alpha_{m,1}\beta_m\widehat{G}^+_{k,m}(\mathbf{z})p_3 \\
(k^2 - \alpha_{m,1}^2)\widehat{G}^-_{k,m}(\mathbf{z})p_1 - \alpha_{m,1}\alpha_{m,2}\widehat{G}^-_{k,m}(\mathbf{z})p_2 + \alpha_{m,1}\beta_m\widehat{G}^-_{k,m}(\mathbf{z})p_3 \\
-\alpha_{m,1}\alpha_{m,2}\widehat{G}^+_{k,m}(\mathbf{z})p_1 + (k^2 - \alpha_{m,2}^2)\widehat{G}^+_{k,m}(\mathbf{z})p_2 - \alpha_{m,2}\beta_m\widehat{G}^+_{k,m}(\mathbf{z})p_3 \\
-\alpha_{m,1}\alpha_{m,2}\widehat{G}^-_{k,m}(\mathbf{z})p_1 + (k^2 - \alpha_{m,2}^2)\widehat{G}^-_{k,m}(\mathbf{z})p_2 + \alpha_{m,2}\beta_m\widehat{G}^-_{k,m}(\mathbf{z})p_3
\end{bmatrix}_{m\in\mathbb{Z}^2}
$$
where $(\widehat{G}^{\pm}_{k,m}(\mathbf{z}))_{m\in\mathbb{Z}^2}$ are the Rayleigh sequences of $G_k(\cdot - \mathbf{z})$ and $\alpha_m$, $\beta_m$ are given by (\ref{alpha}) and (\ref{beta}). 

Denote by $\mathcal{R}(A)$ the range of some operator $A$. The domain $D$ can be characterized by the $\mathcal{R}(H^*)$ 
as follows. 
\begin{theorem} \label{cha2}
A point $\mathbf{z} \in \Omega$ belongs to $D$ if and only if $W(\widehat{\Psi}_{\mathbf{z},m}) \in \mathcal{R}(H^*)$.
\end{theorem}

\begin{proof}
The proof can be done similarly as that of Lemma 7 in~\cite{Lechl2013b}.
\end{proof}

Note that  since the definition of $H^*$ involves $D$, we can't compute $D$ using this characterization. However, this characterization will serve as an intermediate step to connect $D$ to the near field operator $N$ that is given. This is  the 
goal of the Factorization method that is studied in the next section.

\section{The Factorization method}
This section is dedicated to the analysis of the Factorization method. 
Let  $T: L^2(D,\mathbb{C}^3)\times  L^2(D,\mathbb{C}^3)  \to L^2(D,\mathbb{C}^3)\times  L^2(D,\mathbb{C}^3)$ be defined by
$$
T(\mathbf{f},\mathbf{g}) = \begin{bmatrix}Q & -ik\mur^{-1}\xi \\ ik\xi\mur^{-1} & k^2(P-\xi\mur^{-1}\xi) \end{bmatrix} \begin{bmatrix}\mathbf{f}+\curl \mathbf{u}\\ \mathbf{g}+\mathbf{u} \end{bmatrix}
$$
where $\mathbf{u}$ is the radiating variational solution to (\ref{equ}).

\begin{lemma}
$T$ is a bounded  linear operator on $ L^2(D,\mathbb{C}^3)\times  L^2(D,\mathbb{C}^3)$.
\end{lemma}

\begin{proof}
The linearity of $T$ follows from the linearity and well-posedness of (\ref{equ}). We will show its boundedness. For $(\mathbf{f},\mathbf{g}) \in L^2(D,\mathbb{C}^3) \times L^2(D,\mathbb{C}^3)$
$$
\lVert T(\mathbf{f},\mathbf{g}) \rVert^2_{L^2(D,\mathbb{C}^3) \times L^2(D,\mathbb{C}^3)} = \lVert \mathcal{S}(\mathbf{u},\mathbf{f},\mathbf{g}) \rVert_{L^2(D,\mathbb{C}^3)}^2 + \lVert \mathcal{T}(\mathbf{u},\mathbf{f},\mathbf{g}) \rVert_{L^2(D,\mathbb{C}^3)}^2.
$$
Let  $\mathbf{u} \in H_{\alpha,\loc}(\curll,\Omega)$ be the  solution of 
$$
\mathbf{u} = \mathcal{A}\mathcal{S}(\mathbf{u},\mathbf{f},\mathbf{g}) + \mathcal{B}\mathcal{T}(\mathbf{u},\mathbf{f},\mathbf{g}),
$$
which can be rewritten as
\begin{equation}\label{prt}
\mathbf{u} - \mathcal{A}\widetilde{\mathcal{S}}\mathbf{u} - \mathcal{B}\widetilde{\mathcal{T}}\mathbf{u} = \mathcal{A}( (P-\xi\mur^{-1}\xi)\mathbf{g} + \frac{i}{k}\xi\mur^{-1}\mathbf{f}) + \mathcal{B}( Q\mathbf{f} - ik\mur^{-1}\xi\mathbf{g}),
\end{equation}
where
$$
\widetilde{\mathcal{S}}\mathbf{u} =  (P - \xi\mur^{-1}\xi)\mathbf{u} + \frac{i}{k}\xi\mur^{-1}\curl \mathbf{u}, \quad \widetilde{\mathcal{T}}\mathbf{u} = Q\curl \mathbf{u} -ik\mur^{-1}\xi\mathbf{u}.
$$
Since we assume  (\ref{prt}) is well-posed, the operator $I-\mathcal{A}\widetilde{\mathcal{S}} -\mathcal{B}\widetilde{\mathcal{T}}$ is boundedly invertible and
$$
\mathbf{u} = (I-\mathcal{A}\widetilde{\mathcal{S}}-\mathcal{B}\widetilde{\mathcal{T}})^{-1}[ \mathcal{A}( (P-\xi\mur^{-1}\xi)\mathbf{g} + \frac{i}{k}\xi\mur^{-1}\mathbf{f}) + \mathcal{B}( Q\mathbf{f} - ik\mur^{-1}\xi\mathbf{g}) ].
$$
Since $\mathcal{A}, \mathcal{B}$ are bounded and  the scatterer's parameters are 
in $L^{\infty}(\mathbb{R}^3,\C^{3\times 3})$, it follows that there exists $c >0$ such that
$$
\lVert \mathbf{u}\rVert_{H(\curll,D)} \leq c \lVert (\mathbf{f},\mathbf{g}) \rVert_{L^2(D,\mathbb{C}^3) \times L^2(D,\mathbb{C}^3)}.
$$
This implies
\begin{align*}
\lVert \mathbf{g} + \mathbf{u} \rVert_{L^2(D,\mathbb{C}^3)} & \leq (c + 1)\lVert (\mathbf{f},\mathbf{g}) \rVert_{L^2(D,\mathbb{C}^3) \times L^2(D,\mathbb{C}^3)}, \\
\lVert \mathbf{f} + \curl\mathbf{u} \rVert_{L^2(D,\mathbb{C}^3)} & \leq (c + 1)\lVert (\mathbf{f},\mathbf{g}) \rVert_{L^2(D,\mathbb{C}^3) \times L^2(D,\mathbb{C}^3)}.
\end{align*}
Therefore by the definition of $\mathcal{S}$ and $\mathcal{T}$ there exist $c_1,c_2>0$ such that
\begin{align*}
\lVert \mathcal{S}(\mathbf{u},\mathbf{f},\mathbf{g}) \rVert_{L^2(D,\mathbb{C}^3)} &\leq c_1\lVert (\mathbf{f},\mathbf{g}) \rVert_{L^2(D,\mathbb{C}^3) \times L^2(D,\mathbb{C}^3)},\\
\lVert \mathcal{T}(\mathbf{u},\mathbf{f},\mathbf{g}) \rVert_{L^2(D,\mathbb{C}^3)} &\leq c_2\lVert (\mathbf{f},\mathbf{g}) \rVert_{L^2(D,\mathbb{C}^3) \times L^2(D,\mathbb{C}^3)},
\end{align*}
and the boundedness of $T$ follows.
\end{proof}

Let $\Im T$ be  the imaginary part of $T$ defined by
$$
\Im T = \frac{1}{2i}(T - T^*).
$$
In the next lemma we prove the coercivity of $\Im T$ which is the key ingredient  in the analysis of the Factorization method.

\begin{lemma}
\label{coercive}
There exists $c>0$ such that for all $(\mathbf{f},\mathbf{g}) \in L^2(D,\mathbb{C}^3)\times L^2(D,\mathbb{C}^3) $
\begin{equation}\label{ImT}
(\Im T(\mathbf{f},\mathbf{g}),(\mathbf{f},\mathbf{g}))_{L^2(D,\mathbb{C}^3)\times L^2(D,\mathbb{C}^3)} \geq c\lVert (\mathbf{f},\mathbf{g})\rVert^2_{L^2(D,\mathbb{C}^3)\times L^2(D,\mathbb{C}^3)}.
\end{equation}
\end{lemma}

\begin{proof}
For the convenience of the presentation of this proof we will use $(\cdot,\cdot)$ and $\| \cdot \|$ indistinctively for the inner product
and norm of  $L^2(D,\mathbb{C}^3)$ and $L^2(D,\mathbb{C}^3) \times L^2(D,\mathbb{C}^3)$. 
Let $\mathbf{h}_1 = \mathbf{f}+\curl\mathbf{u}$, $\mathbf{h}_2 = \mathbf{g}+\mathbf{u}$ we have
\begin{align*}
(T(\mathbf{f},\mathbf{g}),(\mathbf{f},\mathbf{g})) &= \int_D (Q\mathbf{h}_1 - ik\mur^{-1}\xi\mathbf{h}_2)\cdot(\overline{\mathbf{h}}_1 - \curl\overline{\mathbf{u}})\ \d\mathbf{x}\\
&+ \int_D(ik\xi\mur^{-1}\mathbf{h}_1 + k^2(P-\xi\mur^{-1}\xi)\mathbf{h}_2)\cdot(\overline{\mathbf{h}}_2-\overline{\mathbf{u}})\ \d\mathbf{x} \\
&= S_1 - S_2,
\end{align*}
where
\begin{align*}
S_1 &= (Q\mathbf{h}_1,\mathbf{h}_1) + k^2((P-\xi\mur^{-1}\xi)\mathbf{h}_2,\mathbf{h}_2)  + ik(\xi\mur^{-1}\mathbf{h}_1,\mathbf{h}_2)- ik(\mur^{-1}\xi\mathbf{h}_2,\mathbf{h}_1),\\
S_2 &=\int_D(Q\mathbf{h}_1 - ik\mur^{-1}\xi\mathbf{h}_2)\cdot\curl\overline{\mathbf{u}}\ \d\mathbf{x} +  \int_D(ik\xi\mur^{-1}\mathbf{h}_1+ k^2(P-\xi\mur^{-1}\xi)\mathbf{h}_2)\cdot\overline{\mathbf{u}}\ \d\mathbf{x}.
\end{align*}
Therefore, with the fact that
$$
(\Im T(\mathbf{f},\mathbf{g}),(\mathbf{f},\mathbf{g})) = \Im(T(\mathbf{f},\mathbf{g}),(\mathbf{f},\mathbf{g}))
$$
we have
$$
(\Im T(\mathbf{f},\mathbf{g}),(\mathbf{f},\mathbf{g})) = \Im S_1 - \Im S_2.
$$
Let us first consider $ \Im S_1$. From    Assumption \ref{assum} we have
\begin{align*}
\Im S_1  \geq C_1\|\mathbf{h}_1\|^2 + k^2 C_2\|\mathbf{h}_2\|^2 
+ k\text{Re}(\xi\mur^{-1}\mathbf{h}_1,\mathbf{h}_2) - k\text{Re}(\mur^{-1}\xi\mathbf{h}_2,\mathbf{h}_1).
\end{align*}
Recall that $\mur^{-1}$ and $\xi$ are symmetric. We estimate
\begin{align*}
k\text{Re}(\xi\mur^{-1}\mathbf{h}_1,\mathbf{h}_2) &=\text{Re}(\xi\mur^{-1}\mathbf{h}_1,k\mathbf{h}_2) \\
&= \frac{1}{2}\left(	\lVert\xi\mur^{-1}\mathbf{h}_1 + k\mathbf{h}_2\rVert^2 - \lVert\xi\mur^{-1}\mathbf{h}_1\rVert^2  - k^2\lVert\mathbf{h}_2\rVert^2 \right)\\
&\geq  \frac{1}{2}\left(	\left(\lVert\xi\mur^{-1}\mathbf{h}_1\| -  \|k\mathbf{h}_2\rVert\right)^2 - \lVert\xi\mur^{-1}\mathbf{h}_1\rVert^2  - k^2\lVert\mathbf{h}_2\rVert^2 \right) \\
&= - \lVert\xi\mur^{-1}\mathbf{h}_1\|\|k\mathbf{h}_2\rVert\\
&\geq -\frac{1}{2}\left(\lVert |\xi\mur^{-1}|_F \rVert_{L^\infty}^2 \lVert\mathbf{h}_1\rVert^2 + k^2\lVert\mathbf{h}_2\rVert^2\right)\\
&= -\frac{1}{2}\left(\lVert |\mur^{-1}\xi|_F \rVert_{L^\infty}^2 \lVert\mathbf{h}_1\rVert^2 + k^2\lVert\mathbf{h}_2\rVert^2\right),
\end{align*}
 and similarly
\begin{align*}
-k\text{Re}(\mur^{-1}\xi\mathbf{h}_2,\mathbf{h}_1) &= -\text{Re}(k\mur^{-1}\xi\mathbf{h}_2,\mathbf{h}_1) \\
&= \frac{1}{2}\left(\lVert k\mur^{-1}\xi\mathbf{h}_2 - \mathbf{h}_1\rVert^2 - k^2\lVert\mur^{-1}\xi\mathbf{h}_2\rVert^2 - \lVert\mathbf{h}_1\rVert^2 \right) \\
&\geq -\frac{1}{2}\left(k^2\lVert |\mur^{-1}\xi|_F \rVert_{L^\infty}^2 \lVert\mathbf{h}_2\rVert^2 + \lVert\mathbf{h}_1\rVert^2\right).
\end{align*}
Therefore, we obtain
\begin{align*}
\Im S_1 &\geq \left[C_1 - \frac{1}{2}\left(\lVert |\mur^{-1}\xi|_F\rVert_{L^\infty}^2+1\right)\right] \lVert\mathbf{h}_1\rVert^2 + k^2\left[C_2 - \frac{1}{2}\left(\lVert |\mur^{-1}\xi|_F \rVert_{L^\infty}^2+1\right)\right]\lVert\mathbf{h}_2\rVert^2\\
&= c_1\lVert\mathbf{h}_1\rVert^2+ k^2c_2\lVert\mathbf{h}_2\rVert^2
\end{align*}
where 
\begin{align*}
c_1 &= C_1 - \frac{1}{2}\left(\lVert |\mur^{-1}\xi|_F \rVert_{L^\infty}^2+1\right) >0,\\
c_2 &= C_2 - \frac{1}{2}\left(\lVert |\mur^{-1}\xi|_F \rVert_{L^\infty}^2+1\right) >0.
\end{align*}
In order to estimate $\Im S_2$ we note that since $\mathbf{u}$ is the radiating variational solution to (\ref{equ}), for all $\mathbf{v} \in H_{\alpha}(\curll,\Omega)$ with compact support we have
\begin{align*}
\int_{\Omega}(\curl\mathbf{u}\cdot\curl\overline{\mathbf{v}} - k^2\mathbf{u}\cdot\overline{\mathbf{v}})\ \d\mathbf{x} &= \int_D(Q\mathbf{h}_1 - ik\mur^{-1}\xi\mathbf{h}_2)\cdot\curl\overline{\mathbf{v}}\ \d\mathbf{x} \\
&+  \int_D(ik\xi\mur^{-1}\mathbf{h}_1+ k^2(P-\xi\mur^{-1}\xi)\mathbf{h}_2)\cdot\overline{\mathbf{v}}\ \d\mathbf{x}.
\end{align*}
Let $\Omega_r =\{\mathbf{x}\in\Omega:|x_3|<r\}$ and consider $r>0$  such that $\overline{D} \subset \Omega_r$. Consider a scalar cut-off function $\varphi \in C^{\infty}(\R)$ such that $\varphi(t) = 1$ for $|t| <r$ and $\varphi = 0$ for $|t| >2r$. Then setting $\mathbf{v}(\x) = \varphi(x_3)\mathbf{u}(\x)$ it belongs to $ H_{\alpha}(\curll,\Omega)$ with compact support in $\Omega_{2r}$. Substituting $\v$   in the above variational form yields
$$
\int_{\Omega}(\curl\mathbf{u}\cdot\curl(\varphi\overline{\mathbf{u}}) - k^2\mathbf{u}\cdot(\varphi\overline{\mathbf{u}}))\,\d\mathbf{x} = S_2.
$$
Therefore, using  Green's identities and the fact that $\mathbf{u}$ solves $\curl^2\mathbf{u}-k^2\mathbf{u} = 0$ in $\Omega_{2r}\setminus \Omega_r$ we have
\begin{align*}
S_2 &= \int_{\Omega_r}(|\curl\mathbf{u}|^2 - k^2|\mathbf{u}|^2) \d\mathbf{x} + \int_{\Omega_{2r}\setminus \Omega_r}(\curl\mathbf{u}\cdot\curl(\varphi\overline{\mathbf{u}}) - k^2\mathbf{u}\cdot(\varphi\overline{\mathbf{u}})) \d\mathbf{x} \\
&= \int_{\Omega_r}(|\curl\mathbf{u}|^2 - k^2|\mathbf{u}|^2) \d\mathbf{x}  + \int_{\Omega_{2r}\setminus \Omega_r}(\curl^2\mathbf{u}-k^2\mathbf{u})\cdot(\varphi\mathbf{u}) \d\mathbf{x} \\
 & \qquad + \left(\int_{\{x_3 =  r\} \cap \Omega} - \int_{\{x_3 = - r\} \cap \Omega}\right)( e_3\times \curl\mathbf{u})\cdot\overline{\mathbf{u}}\, \d\mathbf{x}\\
&= \int_{\Omega_r}(|\curl\mathbf{u}|^2 - k^2|\mathbf{u}|^2) \d\mathbf{x} +  \left(\int_{\{x_3 =  r\} \cap \Omega} - \int_{\{x_3 = - r\} \cap \Omega}\right)( e_3\times \curl\mathbf{u})\cdot\overline{\mathbf{u}}\, \d\mathbf{x}.
\end{align*}
Therefore, taking the imaginary part of both sides we have
$$
\Im S_2 =  \Im \left(\int_{\{x_3 =  r\} \cap \Omega} - \int_{\{x_3 = - r\} \cap \Omega}\right)( e_3\times \curl\mathbf{u})\cdot\overline{\mathbf{u}}\, \d\mathbf{x}.
$$
Using the Rayleigh expansion radiation condition for $\u$ and a straightforward calculation  give
$$
 \lim_{r\to\infty} \left(\int_{\{x_3 =  r\} \cap \Omega} - \int_{\{x_3 = - r\} \cap \Omega}\right)( e_3\times \curl\mathbf{u})\cdot\overline{\mathbf{u}}\, \d\mathbf{x}
 = -i4\pi^2\sum_{m: \beta_m>0}\beta_m(|\widehat{u}_m^{+}|^2 + |\widehat{u}_m^{-}|^2).
$$
We thus obtain
$$
\Im S_2 = -4\pi^2\sum_{m: \beta_m>0}\beta_m(|\widehat{u}_m^{+}|^2 + |\widehat{u}_m^{-}|^2) \leq 0.
$$
From the estimates for $\Im S_1$ and $\Im S_2$ and the fact that $(\Im T(\mathbf{f},\mathbf{g}),(\mathbf{f},\mathbf{g})) = \Im S_1 - \Im S_2$, we obtain
\begin{equation}\label{ImT1}
(\Im T(\mathbf{f},\mathbf{g}),(\mathbf{f},\mathbf{g})) \geq c_1 \lVert\mathbf{h}_1\rVert^2 + k^2c_2\lVert\mathbf{h}_2\rVert^2.
\end{equation}
Now suppose that there is no $c>0$ such that (\ref{ImT}) holds. Then there exists a sequence $\{(\mathbf{f}_j,\mathbf{g}_j)\}_j \subset L^2(D,\mathbb{C}^3)\times L^2(D,\mathbb{C}^3)$ such that $\lVert (\mathbf{f}_j,\mathbf{g}_j)\rVert=1$ and
$$
(\Im T(\mathbf{f}_j,\mathbf{g}_j),(\mathbf{f}_j,\mathbf{g}_j)) \xrightarrow{j\to\infty} 0.
$$
By (\ref{ImT1}) we have $\mathbf{h}_1^j = \mathbf{f}_j + \curl\mathbf{u}_j \xrightarrow{j\to\infty} 0$ and $\mathbf{h}_2^j = \mathbf{g}_j + \mathbf{u}_j \xrightarrow{j\to\infty} 0$ in $L^2(D,\mathbb{C}^3)$ where $\mathbf{u}_j$ is the radiating variational solution to 
\begin{equation*}
\curl^2 \mathbf{u}_j-k^2\mathbf{u}_j = k^2\mathcal{S}(\mathbf{u}_j,\mathbf{f}_j,\mathbf{g}_j) + \curl \mathcal{T}(\mathbf{u}_j,\mathbf{f}_j,\mathbf{g}_j).
\end{equation*}
Then $\mathbf{u}_j$ also satisfies
\begin{align*}
\mathbf{u}_j &= \mathcal{A}\mathcal{S}(\mathbf{u}_j,\mathbf{f}_j,\mathbf{g}_j) + \mathcal{B}\mathcal{T}(\mathbf{u}_j,\mathbf{f}_j,\mathbf{g}_j) \\
&= \mathcal{A}\left[\frac{i}{k}\xi\mur^{-1}(\mathbf{f}_j+\curl\mathbf{u}_j) + (P-\xi\mur^{-1}\xi)(\mathbf{g}_j+\mathbf{u}_j)\right] \\
&+ \mathcal{B}\left[Q(\mathbf{f}_j+\curl\mathbf{u}_j) - ik\mur^{-1}\xi(\mathbf{g}_j+\mathbf{u}_j)\right].
\end{align*}
Hence, there exist $c_3,c_4>0$ such that
\begin{align*}
\lVert \mathbf{u}_j \rVert_{H(\curll,D)} &\leq c_3 \left( \frac{1}{k}\lVert |\xi\mur^{-1}|_F \rVert_{L^\infty}\lVert \mathbf{f}_j+\curl\mathbf{u}_j \rVert + \lVert |P-\xi\mur^{-1}\xi|_F \rVert_{L^\infty}\lVert \mathbf{g}_j+\mathbf{u}_j\rVert \right)\\
&+ c_4 \left(\lVert |Q|_F\rVert_{L^\infty}\lVert \mathbf{f}_j+\curl\mathbf{u}_j \rVert + k\lVert | \mur^{-1}\xi |_F \rVert_{L^\infty}\lVert \mathbf{g}_j+\mathbf{u}_j\rVert  \right).
\end{align*}
Thus $\mathbf{u}_j\xrightarrow{j\to\infty}0$ in $H(\curll,D)$ and therefore $(\mathbf{f}_j,\mathbf{g}_j)\xrightarrow{j\to\infty}0$ in $L^2(D,\mathbb{C}^3)\times L^2(D,\mathbb{C}^3)$ which contradicts $\lVert (\mathbf{f}_j,\mathbf{g}_j)\rVert=1$.
\end{proof}

\begin{lemma}\label{fac}
We have the following factorization 
$$
WN = H^*TH.
$$
\end{lemma}
\begin{proof}
Together with $H^* = WE$ from  Lemma~\ref{H*}, the proof follows from the  factorizations $N = GH$ and $G = ET$ that can be easily verified.
\end{proof}
From the above lemma we  obtain
$$
\Im (WN) = H^*(\Im T) H,
$$
and more importantly, we can deduce the following result about the connection between the ranges of $H^*$ and $\Im(WN)$. 

\begin{theorem}\label{cha1}
$\Im (WN)$ is a positive definite, compact and self-adjoint operator on $\ell^2(\mathbb{Z}^2,\mathbb{C}^4)$ with an eigensystem $(\lambda_j,(\phi_m^j))_{j\in\mathbb{N}}$. Thus its square root
$$
(\Im (WN))^{1/2}(\psi_m) =\sum_{j=0}^{\infty}\sqrt{\lambda_j}\ (\psi_m,\phi_m^j)_{\ell^2(\mathbb{Z}^2,\mathbb{C}^4)} \phi^j_m,\quad (\psi_m) \in \ell^2(\mathbb{Z}^2,\mathbb{C}^4),
$$
is well-defined and
$$
\mathcal{R}(H^*) = \mathcal{R}((\Im (WN))^{1/2}).
$$
\end{theorem}

\begin{proof}
$\Im (WN)$ is clearly self-adjoint and its compactness follows from the compactness of $H$. Moreover, for $(a_m) \in \ell^2(\mathbb{Z}^2,\mathbb{C}^4)$ that is a nonzero sequence we have
\begin{align*}
(\Im (WN)(a_m),(a_m))_{\ell^2(\mathbb{Z}^2,\mathbb{C}^4)} &= (H^*(\Im T)H(a_m),(a_m))_{\ell^2(\mathbb{Z}^2,\mathbb{C}^4)} \\
&= ((\Im T)H(a_m),H(a_m))_{L^2(D,\mathbb{C}^3) \times L^2(D,\mathbb{C}^3)}\\
&\geq c\lVert H(a_m) \rVert_{L^2(D,\mathbb{C}^3) \times L^2(D,\mathbb{C}^3)}^2 > 0
\end{align*}
for some $c>0$ thanks to the coercivity of $\Im T$ and injectivity of $H$. Hence all the eigenvalues of $\Im(WN)$ are positive.  The range identity follows directly from the Corollary 1.22 in \cite{Kirsc2008}.
\end{proof}

Combining Theorem \ref{cha2} and Theorem \ref{cha1} we  have the following characterization of $D$.
\begin{theorem}
A point $\mathbf{z}\in \Omega$ belongs to $D$ if and only if $W(\widehat{\Psi}_{\mathbf{z},m}) \in \mathcal{R}((\Im (WN))^{1/2})$. 
If we denote by $(\lambda_j,(\phi_m^j))_{j\in\mathbb{N}}$ the eigensystem of $\Im(WN)$ then the above criterion is equivalent to
\begin{equation}
\label{picard}
 \sum_{j=0}^{\infty} \frac{\left\lvert\left( W(\widehat{\Psi}_{\mathbf{z},m}), (\phi_m^j)\right)_{\ell^2(\mathbb{Z}^2,\mathbb{C}^4)} \right\rvert^2}{\lambda_j} < \infty.
\end{equation}
\end{theorem}
This is a necessary and sufficient characterization for $D$ from the range of $(\Im (WN))^{1/2}$  which gives us the unique determination of $D$. It also provides a fast way to reconstruct $D$ by plotting the  reciprocal value of the series~\eqref{picard} for many points $\z$ sampling some domain that contains $D$.

\section{Numerical examples}

We present in this section some numerical 
examples for imaging of bi-anisotropic periodic structures via the Picard criterion~\eqref{picard}.
For $M  \in \N$, we set
\[
  \Z^2_{M} = \{j  = (j_1,j_2) \in\Z^2: -M/2+1\leq j_1,j_2\leq M/2 \}.
\]
 To generate the synthetic near-field data for the numerical examples 
 we solve the direct problem using the spectral  solver studied in~\cite{Nguye2015}.
More precisely,  the direct problem is solved 
for the incident plane waves $\varphi^{(l)\pm}_j$  for $j\in \Z^2_{M}$, 
and the Rayleigh coefficients of the scattered fields are computed on $x_3 = \pm 1$, again for all 
indices in $\Z^2_{M}$. Let $\bm{N}_{M}$ be the block matrix of
the discretized near-field operator $N$
\begin{equation}
  \label{NM}
  \bm{N}_{M} =  
  \left( \begin{array}{ccccc}
    \left( (\hat{u}^{+}_{1,n})^{(1)+}_j \right)_{j,n} & \left( (\hat{u}^{+}_{1,n})^{(1)-}_j \right)_{j,n} & \left( (\hat{u}^{+}_{1,n})^{(2)+}_j \right)_{j,n} & \left( (\hat{u}^{+}_{1,n})^{(2)-}_j \right)_{j,n} \\
    \left( (\hat{u}^{+}_{2,n})^{(1)+}_j \right)_{j,n} & \left( (\hat{u}^{+}_{2,n})^{(1)-}_j \right)_{j,n} & \left( (\hat{u}^{+}_{2,n})^{(2)+}_j \right)_{j,n} & \left( (\hat{u}^{+}_{2,n})^{(2)-}_j \right)_{j,n} \\
    \left( (\hat{u}^{-}_{1,n})^{(1)+}_j \right)_{j,n} & \left( (\hat{u}^{-}_{1,n})^{(1)-}_j \right)_{j,n} & \left( (\hat{u}^{-}_{1,n})^{(2)+}_j \right)_{j,n} & \left( (\hat{u}^{-}_{1,n})^{(2)-}_j \right)_{j,n} \\
    \left( (\hat{u}^{-}_{2,n})^{(1)+}_j \right)_{j,n} & \left( (\hat{u}^{-}_{2,n})^{(1)-}_j \right)_{j,n} & \left( (\hat{u}^{-}_{2,n})^{(2)+}_j \right)_{j,n} & \left( (\hat{u}^{-}_{2,n})^{(2)-}_j \right)_{j,n} 
  \end{array} \right),
\end{equation}
where the indices $j,n$ in each subblock belong both to $\Z^2_{M}$, and
 $\hat{u}^{\pm}_{(1,2),n}$ are the first two components of the 
Rayleigh coefficients of the scattered field  
in~\eqref{rc}.  The notation 
$( \, \cdot \, )^{(l)\pm}_j$ for $l=1,2$ indicates the dependence of these 
coefficients on the corresponding incident wave $\varphi^{(l)\pm}_j$. 

Let
$\bm{WN}_{M}$ be the discretization of  $WN$. The Hermitian matrix
$\Im(\bm{WN}_{M})$ has an eigendecomposition
$ \Im(\bm{WN}_{M}) = \bm{V} \bm{D} \bm{V}^{-1}$,
where $\bm{D}$ is the diagonal matrix containing  
$4M^2$ 
eigenvalues $\bm{\lambda}_n$ of
$\Im(\bm{WN}_{M})$ and $\bm{V}$ is an orthogonal matrix containing the eigenvectors $(\bm{\varphi}_{j,n})_{j=1}^{4M^2}$.  Then
\begin{equation}
  \label{discr12}
( \Im(\bm{WN}_{M}))^{1/2} = \bm{V} \, | \bm{D}|^{1/2} \, \bm{V}^{-1}.
\end{equation}
Then the criterion~\eqref{picard} is numerically exploited for imaging 
by plotting the function 
\begin{align}
\label{discreteCriterion}
 \z \mapsto P_{M}(\z) 
 = \Bigg[\sum_{n=1}^{4M^2} \frac{|A_n(\z)|^2}{\bm{\lambda}_n}\Bigg]^{-1} ,
 \end{align}
where 
 $A_n(\z) = \sum_{j=1}^{4M^2} \bm{W}(\widehat{\Psi}_{\z,j}) \overline{\bm{\phi}_{j,n}}.$
If the series in~\eqref{discreteCriterion} approximates the true value of 
the exact Picard series in~\eqref{picard}, then $P_{M}$ 
should be very small outside of $D$ and considerably larger inside $D$. 

To consider noise in the scattering data
 we add a complex-valued noise matrix $\bm{X}$ 
containing random  numbers whose real and imaginary parts are uniformly distributed on $(-1,1)$ to the data matrix $\bm{N}_{M}$.
Denoting by $\delta$ the noise level, the
noisy data matrix $(\bm{N}_{M})_{\delta}$ is then given by
\begin{align*}
 (\bm{N}_{M})_{\delta} 
 = \bm{N}_{M}
   + \delta\frac{\bm{X}}{\|\bm{X}\|} \left\|\bm{N}_{M} \right\|,
\end{align*}
where the matrix norm $\|\cdot\|$ is the Frobenius norm. 
For such noisy  data, the eigenvalue decomposition 
in~\eqref{discr12} has to be replaced by a singular value decomposition, 
that  will not be detailed here. We truncate the singular values to regularize 
the Factorization method. For all of the examples below we only keep the singular
values that are greater than or equal to $10^{-2}$ for the regularization. 

We consider four numerical examples for which the periodic structures are motivated by two-dimensional
photonic crystals.  Here are the detailed information of the periodic structures 
we consider in this section.

a)  We consider the  structure of periodically aligned balls. The reconstruction result for this example is presented in Figure~\ref{fig1}. Recall that $\ol{D} = [\mathrm{supp}(Q) \cup \mathrm{supp}(P)] \cap \Omega$. In this example $D$ is given by
\begin{align*}
D &= \left\{(x_1,x_2,x_3)^\top:\left(x_1-\frac{\pi}{2}\right)^2 + \left(x_2-\frac{\pi}{2}\right)^2 + x_3^2 < 0.6^2 \right\} \\
&\cup \left\{(x_1,x_2,x_3)^\top:\left(x_1+\frac{\pi}{2}\right)^2 + \left(x_2-\frac{\pi}{2}\right)^2 + x_3^2 < 0.6^2 \right\} \\
&\cup \left\{(x_1,x_2,x_3)^\top:\left(x_1-\frac{\pi}{2}\right)^2 + \left(x_2+\frac{\pi}{2}\right)^2 + x_3^2 < 0.6^2 \right\} \\
&\cup \left\{(x_1,x_2,x_3)^\top:\left(x_1+\frac{\pi}{2}\right)^2 + \left(x_2+\frac{\pi}{2}\right)^2 + x_3^2 < 0.6^2 \right\}.
\end{align*}

b)  We consider the  structure of periodically aligned bars. The reconstruction result for this example is presented in Figure~\ref{fig2}. In this example $D$ is given by
\begin{align*}
D &= \left\{(x_1,x_2,x_3)^\top:x_1^2 + x_3^2 < \left(\frac{\pi}{6}\right)^2 \right\} \\
&\cup \left\{(x_1,x_2,x_3)^\top:\left(x_1-\pi\right)^2 + x_3^2 < \left(\frac{\pi}{6}\right)^2 \right\} \\
&\cup \left\{(x_1,x_2,x_3)^\top:\left(x_1+\pi\right)^2 + x_3^2 < \left(\frac{\pi}{6}\right)^2 \right\}.
\end{align*}

c)   We consider the  structure of periodically aligned cubes. The reconstruction result for this example is presented in Figure~\ref{fig3}. In this example $D$ is given by
$$
D = \left\{(x_1,x_2,x_3)^\top: |x_1|<\frac{\pi}{2},\ |x_2|<\frac{\pi}{2},\ |x_3|<0.3 \right\}.
$$

d) We consider a strip with periodically aligned holes. The reconstruction result for this example is presented in Figure~\ref{fig4}. In this example $D$ is given by
$$
D = \left\{(x_1,x_2,x_3)^\top: x_1^2 + x_2^2 > \left(\frac{\pi}{2}\right)^2,\ |x_3|<0.3 \right\}.
$$

The numerical implementation  is done using Matlab. In  all of the examples in this section  we use the following parameters.
\begin{align*}
\text{sampling domain}= (-\pi,\pi)^2\times (-1,1),\quad  k = \pi,\\
 \quad M = 20 \text{ (i.e. 1600 incident plane waves)},  \\
\quad \delta = 2\% \text{ (noise level)}, \quad \alpha = (\pi/2, \pi/2, 0).
\end{align*}
The sampling domain is probed by $32^3$ sampling points.  
Recall that the data are measured at $\{x_3 = \pm 1\}$.
The  matrix-valued coefficients $\epsr, \mur^{-1}, \xi$ are given 
by~\eqref{coeff1}--\eqref{coeff2} in all of the examples. The Rayleigh expansion in the radiation condition~\eqref{rc} for $\u$ is truncated in $\Z_M^2$. Hence, for $M = 20$, we have 400 Rayleigh coefficients in each block of the data matrix~\eqref{NM}.
There are 32 coefficients corresponding to   propagating modes in these 400 Rayleigh  coefficients, and the rest corresponds
to evanescent modes which are necessarily important for the quality of the reconstructions. The important role of  evanescent modes in the numerical implementation has also been observed in previous works, see for example~\cite{Arens2005, Lechl2013b, Jiang2017}.
We can see in the Figures~\ref{fig1},\ref{fig2},\ref{fig3}, and \ref{fig4} that the Factorization method is able to provide reasonable reconstructions
for different types of bi-anisotropic periodic structures. 

 \begin{figure}[ht!]
\centering
\subfloat[Exact geometry viewed at $\Omega\cap \{z=0\} $]{\includegraphics[width=4.5cm]{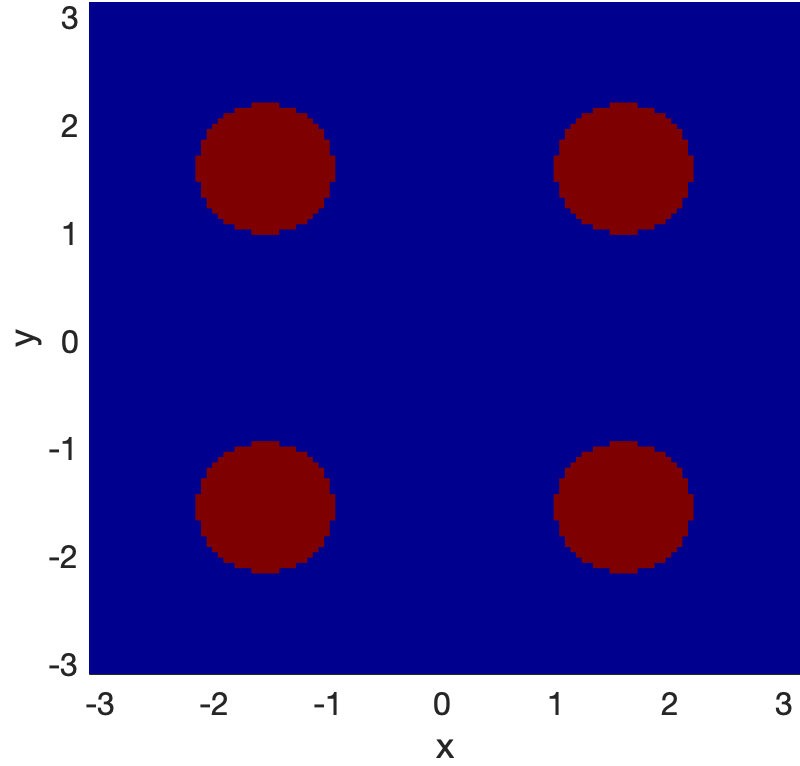}} \hspace{2cm}
\subfloat[Reconstruction  viewed at $\Omega \cap \{z=0\}$]{\includegraphics[width=4.5cm]{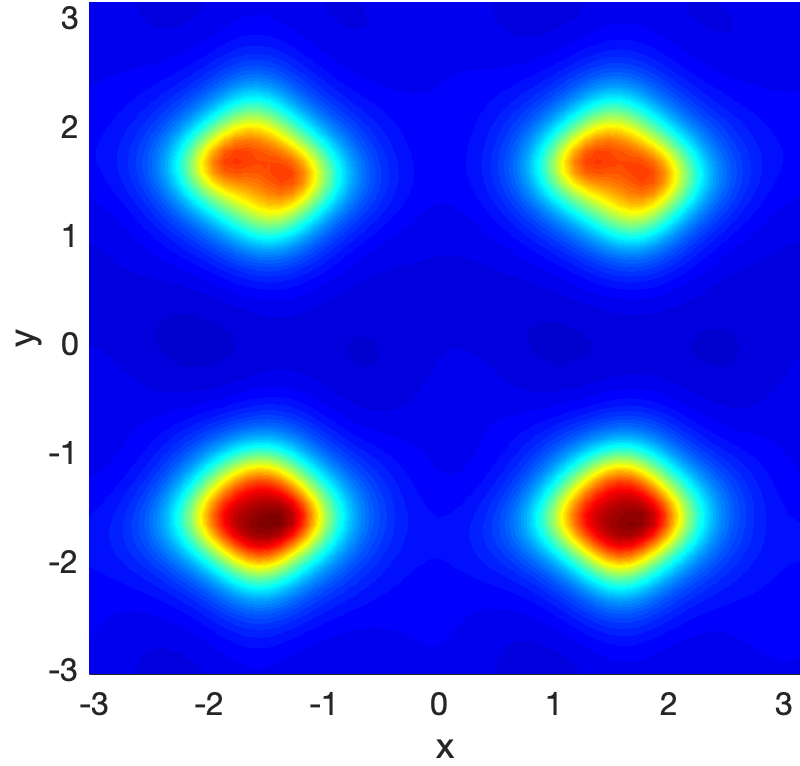}}\\
\subfloat[Exact geometry in $(-3\pi, 3\pi)^2\times (-2,2)$]{\includegraphics[width=7.5cm]{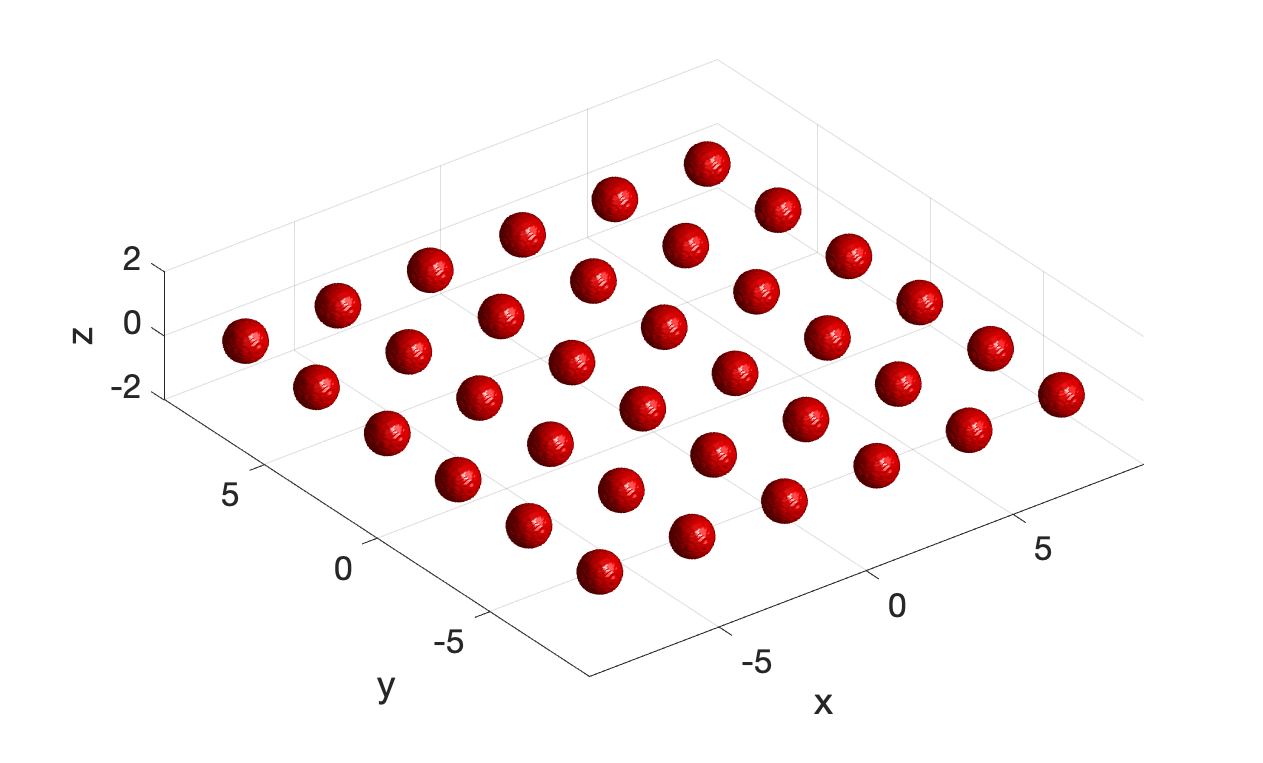}} \hspace{-0.5cm}
\subfloat[Reconstruction]{\includegraphics[width=7.5cm]{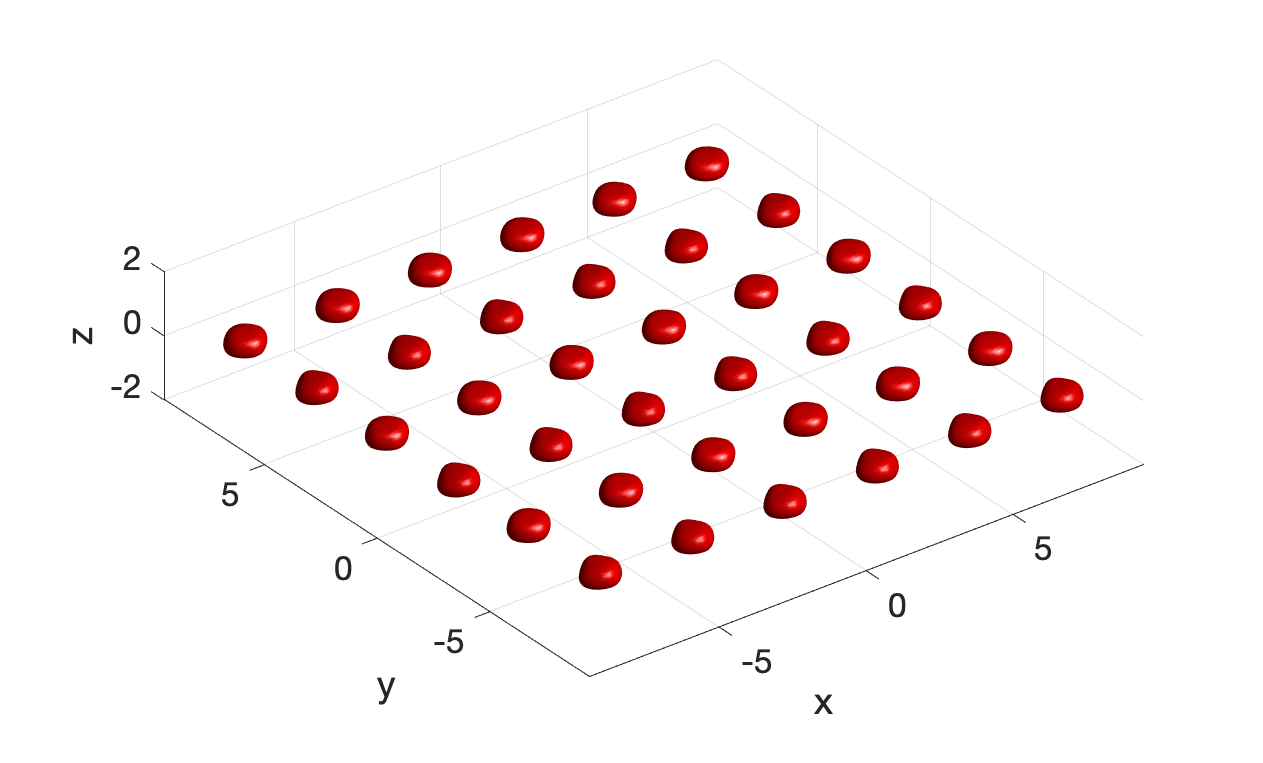}}\\
\caption{Shape reconstruction of  periodically aligned balls. There is  2$\%$ artificial noise in the data.
The isovalue for the isosurface plotting is chosen to be one third of the maximal value of the computed image. }
\label{fig1}
\end{figure}

 \begin{figure}[h!]
\centering
\subfloat[Exact geometry viewed at $\Omega\cap \{z=0\}$ ]{\includegraphics[width=4.5cm]{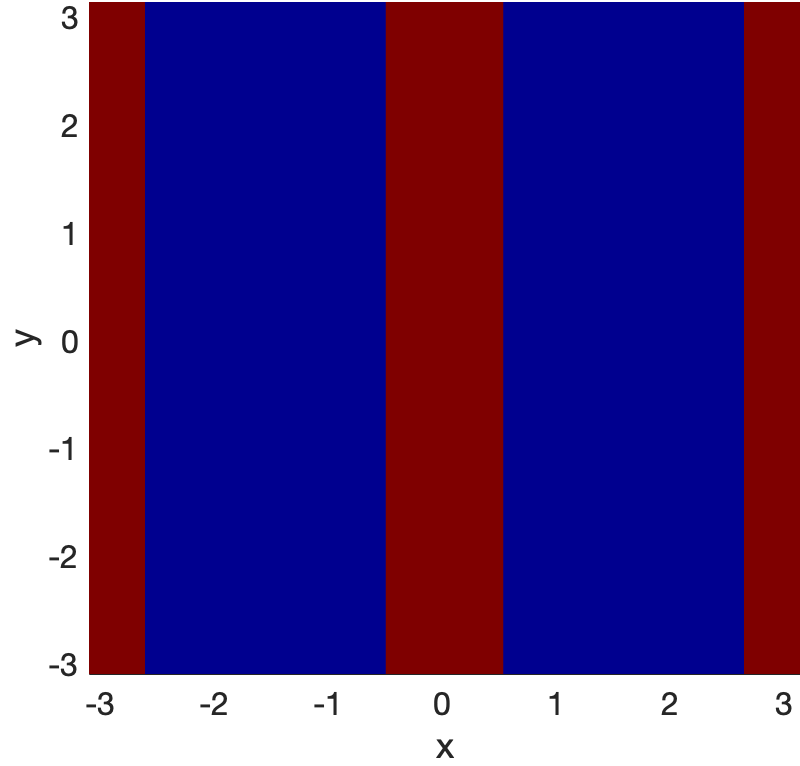}} \hspace{2cm}
\subfloat[Reconstruction viewed at $\Omega\cap \{z=0\}$ ]{\includegraphics[width=4.5cm]{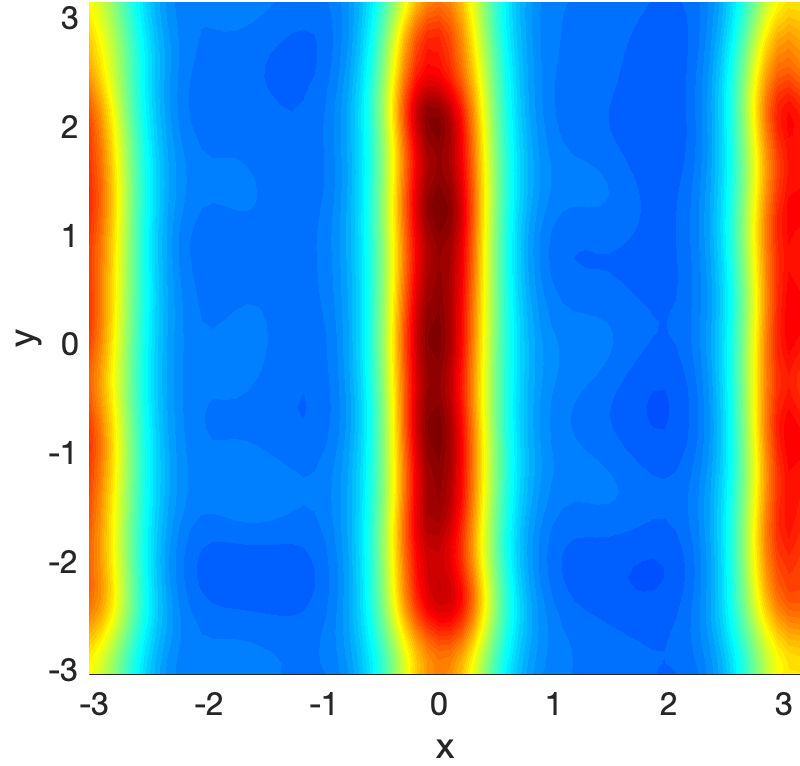}}\\
\subfloat[Exact geometry in $(-3\pi, 3\pi)^2\times (-2,2)$]{\includegraphics[width=7.5cm]{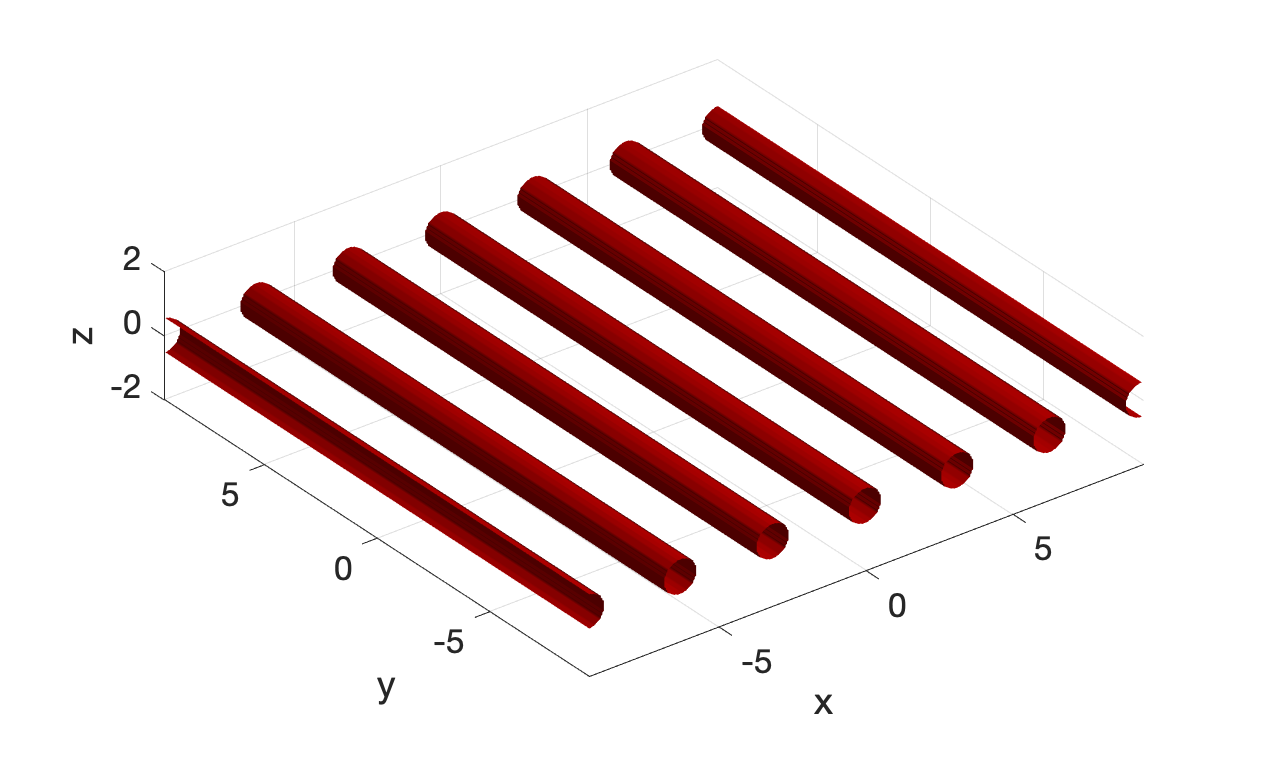}} \hspace{-0.5cm}
\subfloat[Reconstruction]{\includegraphics[width=7.5cm]{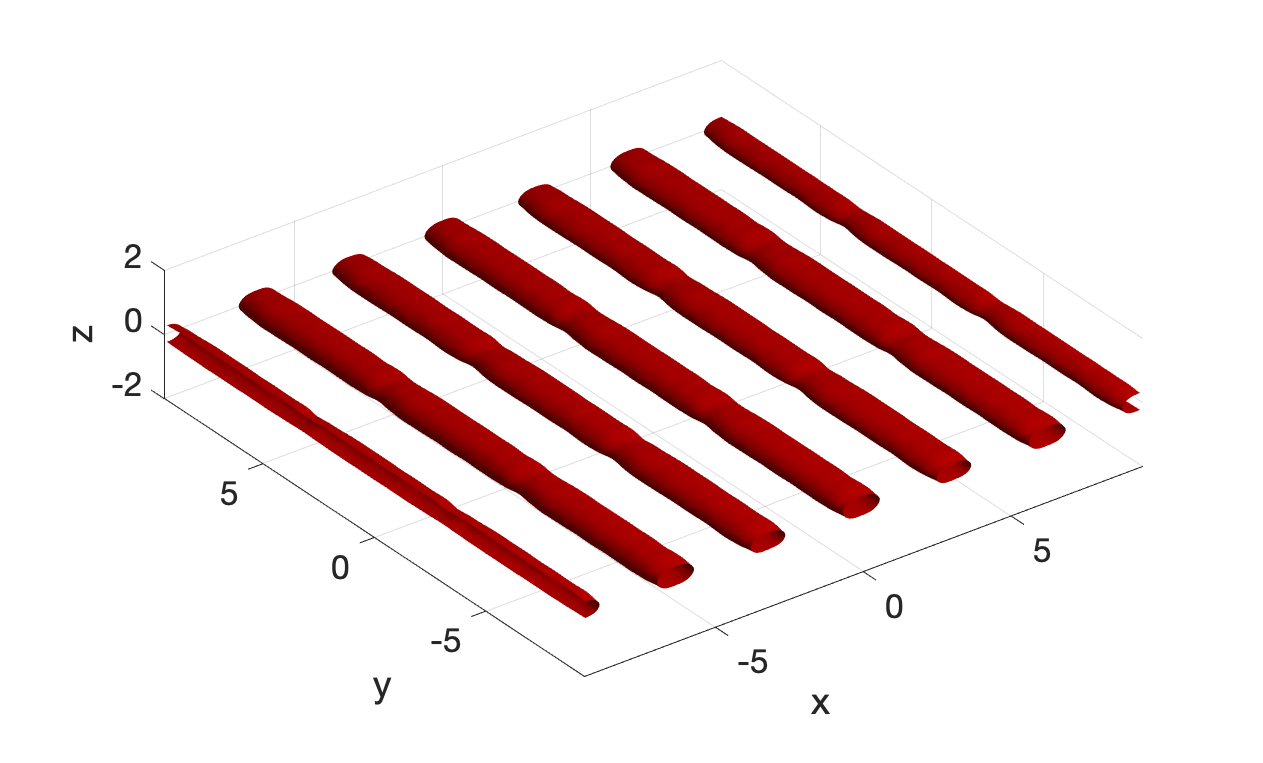}}\\
\caption{Shape reconstruction of periodically aligned bars. There is  2$\%$ artificial noise in the data.
The isovalue for the isosurface plotting is chosen to be one third of the maximal value of the computed image.}
\label{fig2}
\end{figure}

 \begin{figure}[h!]
\centering
\subfloat[Exact geometry viewed at $\Omega\cap \{z=0\}$]{\includegraphics[width=4.5cm]{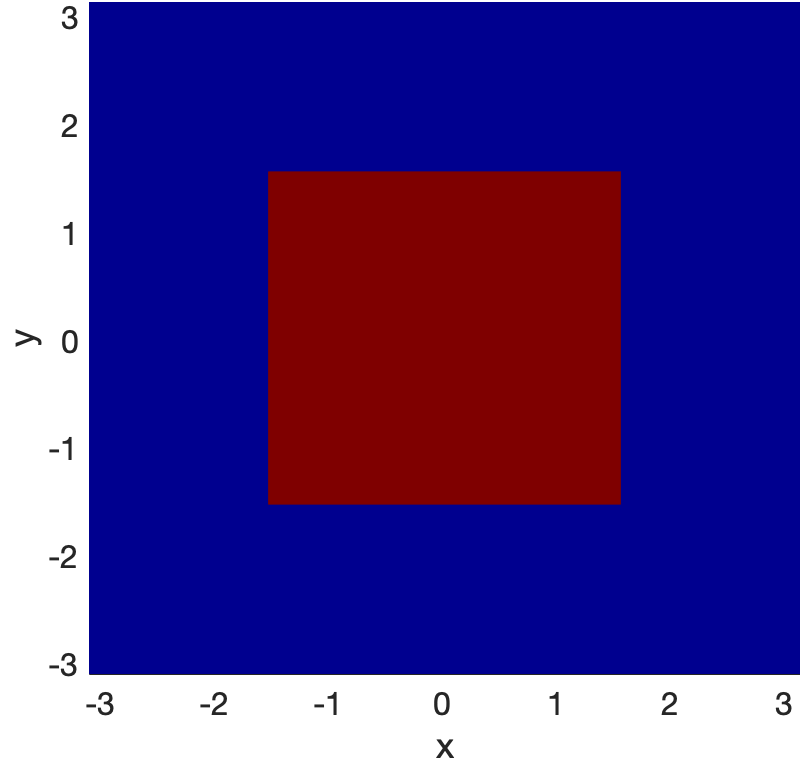}} \hspace{2cm}
\subfloat[Reconstruction  viewed at $\Omega\cap \{z=0\}$]{\includegraphics[width=4.5cm]{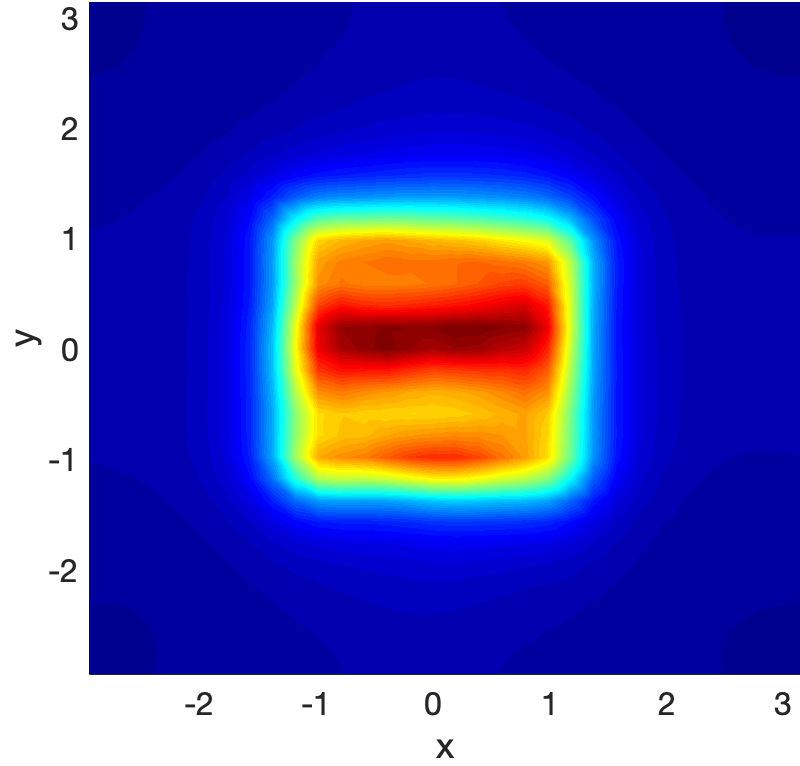}}\\
\subfloat[Exact geometry in $(-3\pi, 3\pi)^2\times (-2,2)$]{\includegraphics[width=7.5cm]{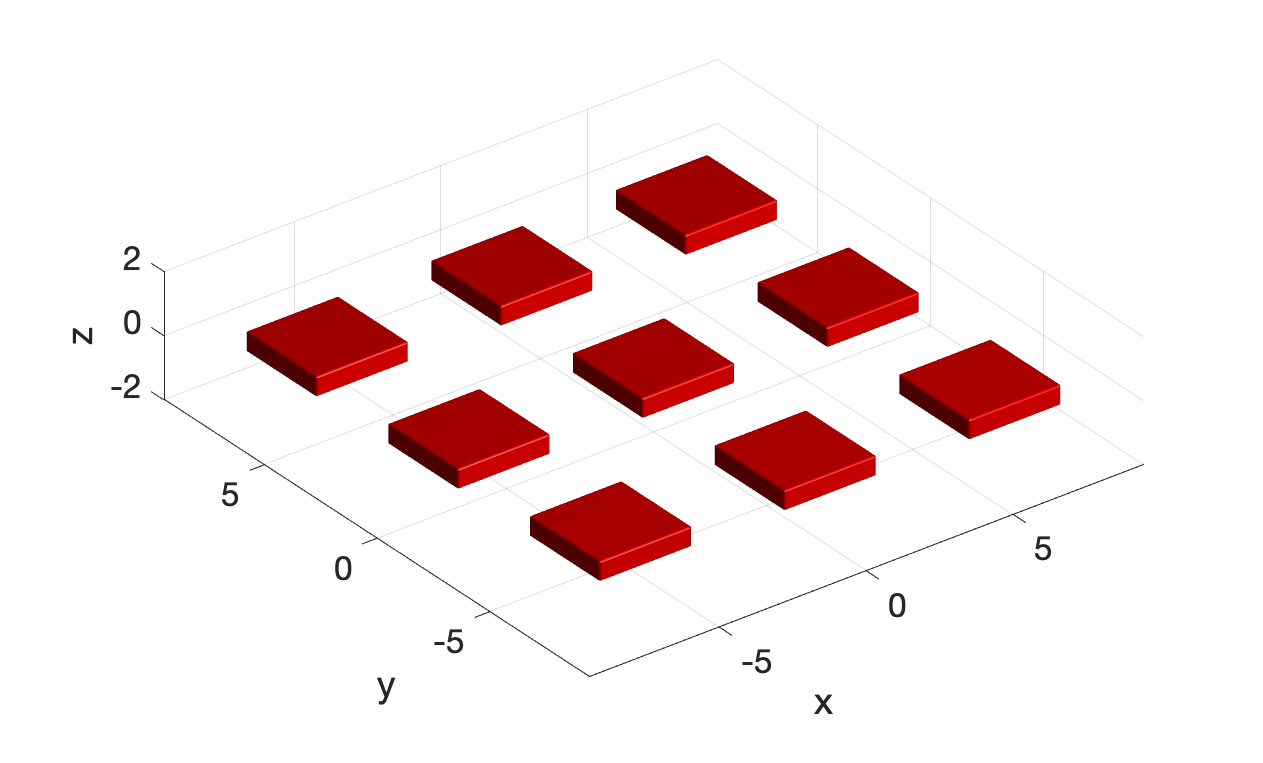}} \hspace{-0.5cm}
\subfloat[Reconstruction]{\includegraphics[width=7.5cm]{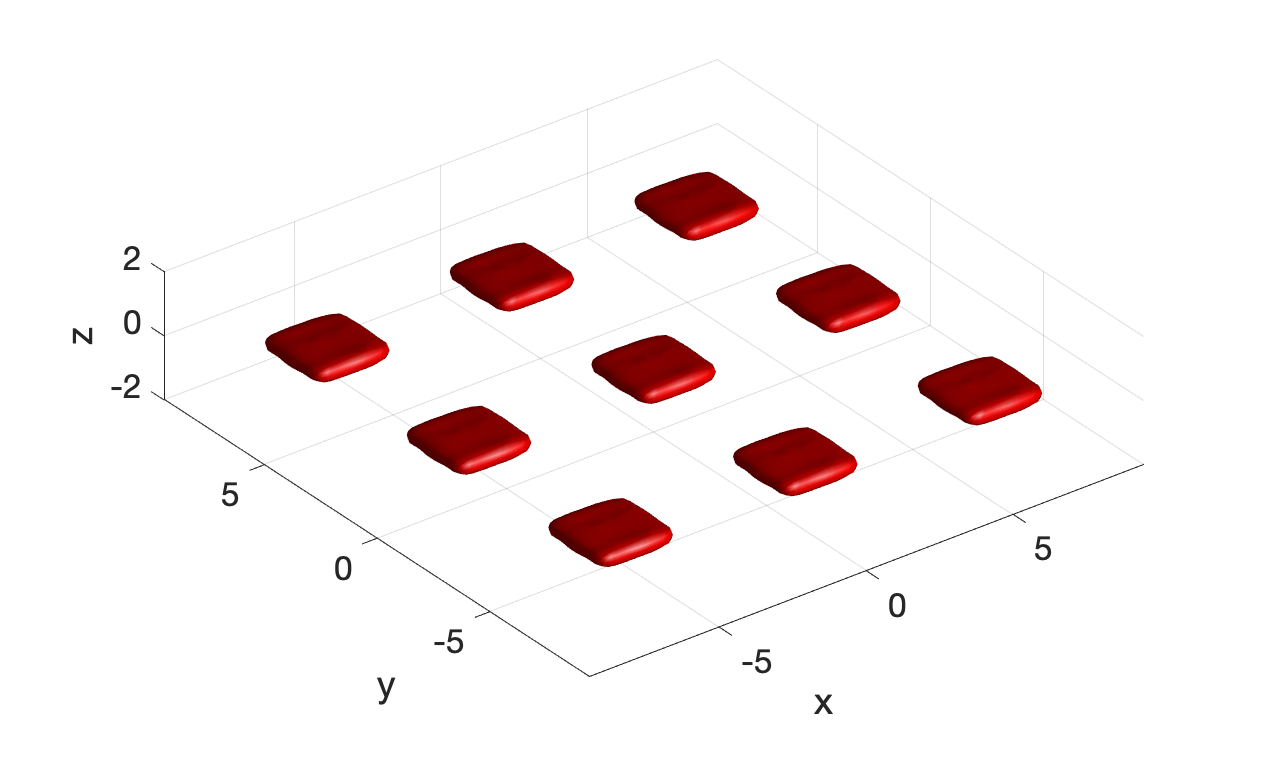}}\\
\caption{Shape reconstruction of  periodically aligned cubes.  There is  2$\%$ artificial noise in the data.
The isovalue for the isosurface plotting is chosen to be one third of the maximal value of the computed image.}
\label{fig3}
\end{figure}

 \begin{figure}[ht!]
\centering
\subfloat[Exact geometry viewed at $\Omega\cap \{z=0\}$]{\includegraphics[width=4.5cm]{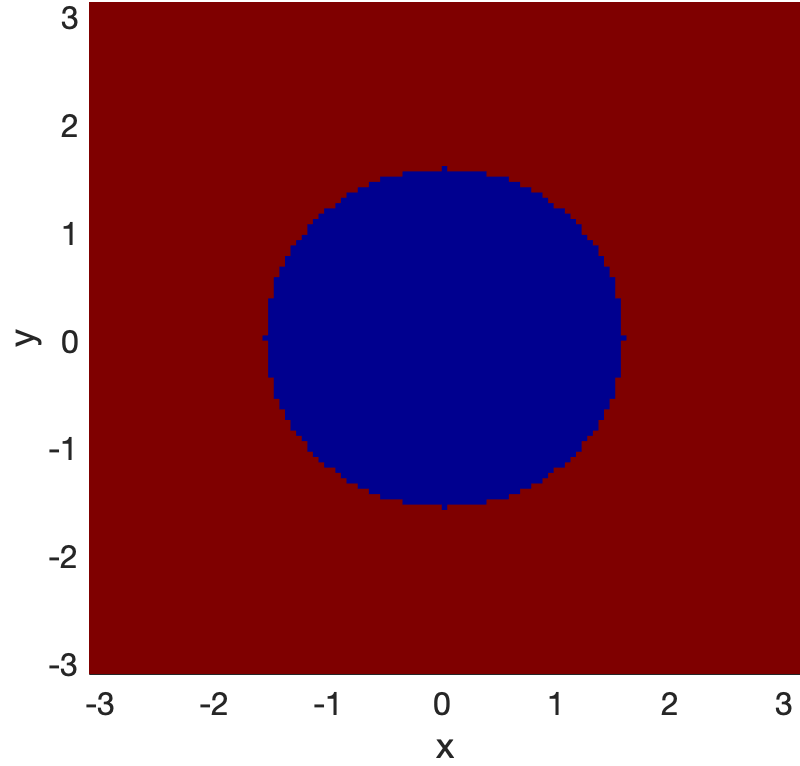}} \hspace{2cm}
\subfloat[Reconstruction viewed at $\Omega\cap \{z=0\}$]{\includegraphics[width=4.5cm]{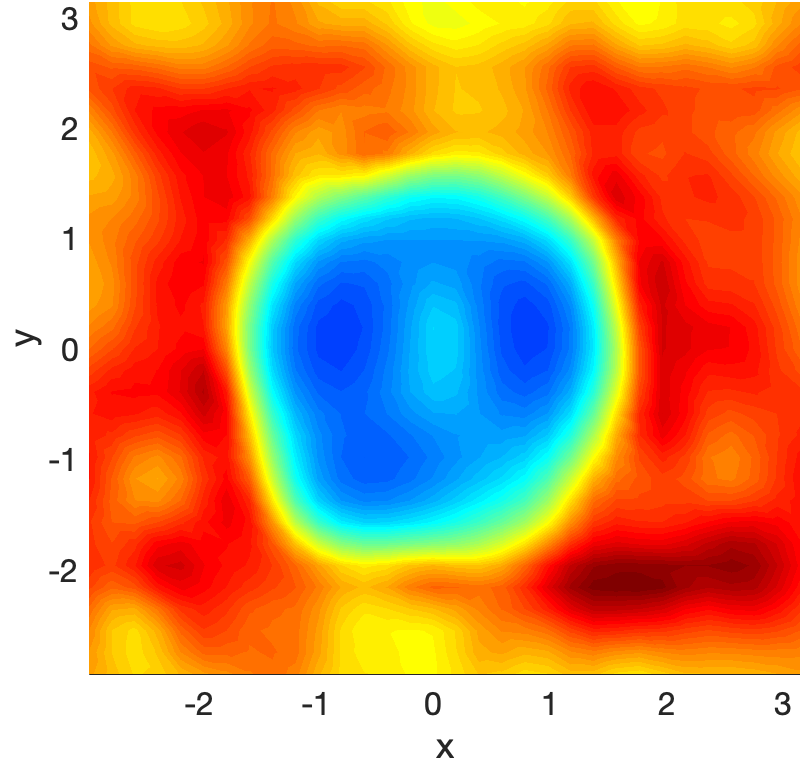}}\\
\subfloat[Exact geometry in $(-3\pi, 3\pi)^2\times (-2,2)$]{\includegraphics[width=7.5cm]{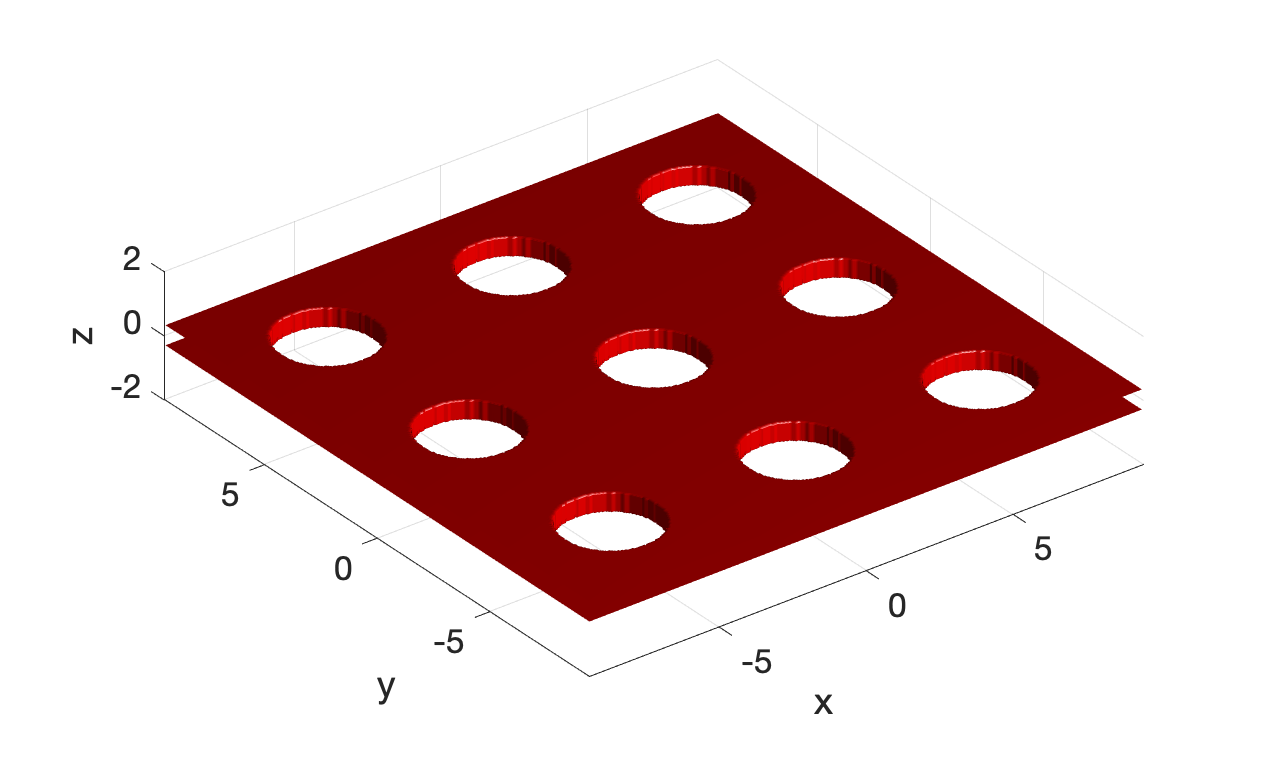}} \hspace{-0.5cm}
\subfloat[Reconstruction]{\includegraphics[width=7.5cm]{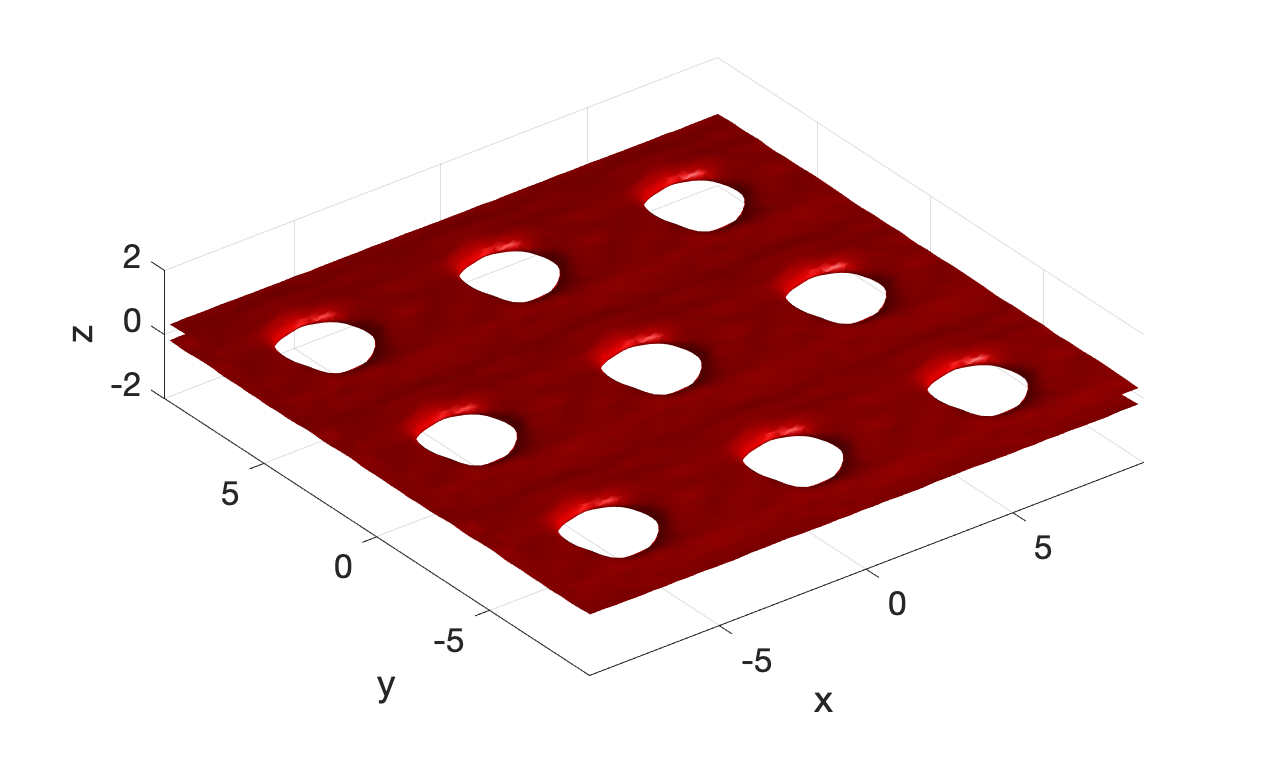}}\\
\caption{Shape reconstruction of a strip with periodically aligned holes.  There is  2$\%$ artificial noise in the data.
The isovalue for the isosurface plotting is chosen to be one third of the maximal value of the computed image.}
\label{fig4}
\end{figure}

\vspace{0.5cm}

{\bf Acknowledgement.} The authors were partially supported by NSF grant DMS-1812693.
\bibliographystyle{plain}
\bibliography{ip-biblio2}

\end{document}

%% file: mathmacros.tex
\newcommand{\N}{\ensuremath{\mathbb{N}}}
\newcommand{\Z}{\ensuremath{\mathbb{Z}}}
\newcommand{\R}{\ensuremath{\mathbb{R}}}
\newcommand{\C}{\ensuremath{\mathbb{C}}}

\renewcommand{\d}{\mathrm{d}}

\newcommand{\ol}{\overline}

\newcommand{\mur}{\mu_{\mathrm{r}}}
\newcommand{\loc}{\mathrm{loc}}

\renewcommand{\Re}{\mathrm{Re}\,}
\renewcommand{\Im}{\mathrm{Im}\,}

\newcommand{\epsr}{\varepsilon_\mathrm{r}}

\newcommand{\E}{{\mathbf{E}}}

\renewcommand{\div}{\mathrm{div}}
\newcommand{\curl}{\mathrm{curl}\,}
\newcommand{\curll}{\mathrm{curl}}

\newcommand{\x}{\mathbf{x}}

\newcommand{\z}{\mathbf{z}}

\renewcommand{\u}{\mathbf{u}}
\renewcommand{\v}{\mathbf{v}}